\documentclass[leqno,11pt,letterpaper, english]{amsart}
\usepackage[usenames,dvipsnames]{color}
\usepackage[colorlinks=true,linkcolor=Red,citecolor=Green]{hyperref}
\usepackage{amsmath,amssymb,amsthm,graphicx,url}
\usepackage{epstopdf}
\usepackage{color}
\usepackage{dsfont}
\usepackage[T1]{fontenc}
\usepackage{natbib}
\usepackage{comment}

\newcommand{\xqedhere}[1]{%
	\rlap{%
		\hbox to#1{%
			\hfil
			\llap{%
				\ensuremath{\square}
			}%
		}%
	}%
}

\def\pasdegrille{\let\grille = \pasgrille}

\def\aat#1#2#3{
	\divide \dimen1 by 48 \dimen3=\dimen1 \multiply \dimen1 by #1
	\advance \dimen1 by -\dimen3 \divide \dimen1 by 101 \multiply
	\dimen1 by 100 \divide \dimen2 by \count11 \multiply \dimen2 by #2
	\setbox0=\hbox{#3}\ht0=0pt\dp0=0pt
	\rlap{\kern\dimen1 \vbox to0pt{\kern-\dimen2\box0\vss}}\dimen1= \wd1
	\dimen2=\ht1}
\def\pasgrille{
	\count12= \dimen1 \divide \count12 by 50 \divide \dimen2 by
	\count12 \count11 =\dimen2 \ \divide \dimen1 by 48
	\setlength{\unitlength}{\dimen1} \smash{\rlap{\ }} \dimen1= \wd1
	\dimen2=\ht1 }
\def\grille{
	\count12= \dimen1 \divide \count12 by 50 \divide \dimen2 by
	\count12 \count11 =\dimen2 \ \divide \dimen1 by 48
	\setlength{\unitlength}{\dimen1}
	\smash{\rlap{\graphpaper[1](0,0)(50, \count11)}} \dimen1= \wd1
	\dimen2=\ht1 }

\pasdegrille

\newcommand{\be}{\begin{equation}}
	\newcommand{\ee}{\end{equation}}

\newcommand{\D}{{\mathcal D}}

\newcommand{\nn}{\nonumber} 

\theoremstyle{plain}

\newtheorem{thm}{Theorem}

\newtheorem{prop}{Proposition}[section]

\newtheorem{lem}[prop]{Lemma}

\theoremstyle{definition}

\newtheorem{defn}[prop]{Definition} 

\numberwithin{equation}{section}

\def\squarebox#1{\hbox to #1{\hfill\vbox to #1{\vfill}}}

\usepackage{amsxtra}

\ifx\pdfoutput\undefined
\DeclareGraphicsExtensions{.pstex, .eps}
\else
\ifx\pdfoutput\relax
\DeclareGraphicsExtensions{.pstex, .eps}
\else
\ifnum\pdfoutput>0
\DeclareGraphicsExtensions{.pdf}
\else
\DeclareGraphicsExtensions{.pstex, .eps}
\fi
\fi
\fi

\title[Semiclassical localization of Schrödinger's eigenfunctions]{Semiclassical localization of Schrödinger's eigenfunctions}

\author[S.~Campagne]{Sébastien Campagne}
\address{Universit{\'e} Paris-Saclay, Math{\'e}matiques, UMR 8628 du CNRS, B{\^a}t 307, 91405  Orsay Cedex, France,   and Institut Universitaire de France}
\email{Sebastien.campagne@universite-paris-saclay.fr}

\date{\today}

\usepackage{amssymb}
\usepackage{amsmath, amsthm, amsopn, amsfonts}

\def\11{{\rm 1~\hspace{-1.4ex}l} }
\def\R{\mathbb R}

\def\S{\mathbb S}
\def\H{\mathbb H}

\begin{document}    

\begin{abstract}
This article addresses the microlocalization of eigenfunctions for the semiclassical Schr\"odinger operator $-h^2\Delta+V$ on closed Riemann surfaces with real bounded potentials. Our primary aim is to establish quantitative bounds on the spatial concentration of these eigenfunctions, extending classical results, typically restricted to smooth potentials, to the more general case where the potential is merely bounded.


Our main result provides an explicit exponential bound for the $L^2$-norm of eigenfunctions on the entire surface in terms of their $L^2$-norm on an arbitrary open subset with an exponential weight of $Ch^{-1}\log(h)^2$. This bound improves upon previous estimates for non-smooth potentials that was an exponential weight of $Ch^{-4/3}$. Our proof is based  on a recent approach of the Landis conjecture develop by \cite*{logunov2020landisconjectureexponentialdecay}.
\end{abstract}

\ \vskip -1cm \noindent\hfil\rule{0.9\textwidth}{.4pt}\hfil \vskip 1cm 
	\maketitle
	

\tableofcontents

\section{Introduction}	
	In this article, we study the localization properties of eigenfunctions $u_h$ for the semiclassical Schrödinger operator $P_V$ defined on a Riemann surface $(M,g)$. Specifically, we consider the equation:

	\begin{equation}
		\label{eq: Schrodinger equation}
		(P_V-E)u_h:= (-h^2\Delta_g +V- E)u_h=0
	\end{equation}

	where $E\in I$ is an energy level ranging over a compact set $I\subset \R$ and $V\in L^\infty(M,\R)$ is a real-valued bounded potential. Our goal is to understand how the geometry of $M$ and the regularity of the potential $V$ affect the spatial concentration of $u_h$ in the semiclassical limit $h \rightarrow 0$.




	When the potential $V$ is smooth, it is well known that one can control the global behaviour of eigenfunctions $u_h$ using their values on any open subset $U \subset M$. More precisely, for any energy $E \in I$, there is a constant $C>0$ such that, for $h$ small enough, any eigenfunction $u_h$ of $P_V$ with energy $E$ satisfies:


	\begin{equation}
		\label{eq: controlefonctpropreVlisse}
		\int_M \Big(|u_h|^2 + |h\nabla u_h|^2 \Big)d\mu_g \leq e^{C/h} \int_U \Big(|u_h|^2 + |h\nabla u_h|^2 \Big)d\mu_g.
	\end{equation}
	where $\mu_g$ is the volume form on $(M,g)$.
	 This estimate relies on analytic tools, in particular Carleman estimates. More precisely, for any fixed radius $R > 0$ and energy $E \in \R$, one has

	\begin{equation}
		\int_{B(0, R)} \Big(|u|^2 + |h\nabla u|^2 \Big)dx \leq  e^{C/h} \int_{B(0, R)} |(P_V - E)u|^2dx
	\end{equation}
	for  $h$ small enough and for all $u$ smooth function with compact support in the ball $B(0,R)$ of the plan $\R^2$. See, for instance \cite{lerousseauCarlemanEstimatesElliptic2012} or in the book \cite*{LeRousseau2022}.
	
What can we say when the potential is not smooth?
In dimension one, we can recover the same result when $V\in L^\infty$, see for instance \cite{dyatlovMathematicalTheoryScattering2019} in the proof of theorem 2.32. We can do even better for certain classes of potentials thanks to the Agmon distance  \citep{laurentUniformObservationSemiclassicalSchrodinger2023}.
	In dimension $\geq 2$, Lipschitz regularity is sufficient for \eqref{eq: controlefonctpropreVlisse}; see for example \cite{datchev2014quantitativelimitingabsorptionprinciple}, \cite{shapiro2017semiclassicalresolventboundsdimension} or \cite{dyatlovMathematicalTheoryScattering2019}. In the general case, \cite{kloppSemiclassicalResolventEstimates2019a} showed that if $V \in L^{\infty}$, then

	\begin{equation}
		\label{eq:KloppVogel}
		\int_{B(0, R)} \Big(|u|^2 + |h\nabla u|^2 \Big)dx \leq  e^{C/h^{4/3}} \int_{B(0, R)} |(P_V - E)u|^2dx.
	\end{equation}
	for  $h$ small enough and for all $u$ smooth function with compact support in $B(0,R)$. This result was then improved by \cite{vodevSemiclassicalResolventEstimates2020} for Hölder potentials.
	 In \cite{Campagne2026a}, inspired by the work of \cite{vodevSemiclassicalResolventEstimates2020},  we proved a new estimate for eigenfunctions on Riemann surfaces that takes into account the regularity of $V$: In the same conditions as \ref{eq: controlefonctpropreVlisse} with $V$ bounded and not necessarily smooth, one has

		\begin{equation}
			\int_M \Big( |u_h|^2 + |h\nabla u_h|^2 \Big)d\mu_g \leq  e^{C\beta(h)} \int_U\Big( |u_h|^2 + |h\nabla u_h|^2\Big)d\mu_g,
		\end{equation}
		where
		\begin{equation}
			\beta(h) = \frac{1}{h^{4/3}} \sup_{x_0 \in M} \sup_{x \in B(x_0, h^{2/3} \kappa)} |V(x) - V(x_0)|^{1/2},
		\end{equation}
		and $\kappa > 0$ is a fixed small constant.  This result gives us a continuous control between Lipschitz potentials ($\beta(h)\simeq h^{-1}$)  and only bounded potentials ($\beta(h)\simeq h^{-4/3}$).

	Our goal in this article is to improve \eqref{eq:KloppVogel} by using a recent approach developed by \cite* {logunov2020landisconjectureexponentialdecay} to study a similar problem. We show the following:

	\begin{thm}
		Let $(M,g)$ be a closed Riemann surface with the volume form $\mu_g$, and let $U \subset M$ be an open subset. Let $E \in I\subset \R$, and assume that $u_h \in L^2(M,\R)$ satisfies
		\[
		(-h^2\Delta_g +V- E)u_h = 0,
		\]
		with $V \in L^\infty(M, \R)$. Then there exist constants $C > 0$ and $h_0 > 0$, independent of $u_h$, such that for all $0 < h < h_0$, we have
		\begin{equation}
			\int_M \Big(|u_{h}|^2+ |h\nabla u_{h}|^2\Big)d\mu_g    \leq \exp\left( C \frac{\log(h)^{2}}{h} \right) \int_U \Big(|u_{h}|^2+ |h\nabla u_{h}|^2 \Big)d\mu_g.
		\end{equation}
	\end{thm}
We note that the exponent $2$ on the logarithm could probably be improved with further work. 


		%
			%





The work of \cite*{logunov2020landisconjectureexponentialdecay} was originally focus on the Landis conjecture.
	As a reminder, the Landis conjecture (see \cite{Kondratev1991}) deals with how quickly solutions $u \neq 0$ to

	\begin{equation}
		(-\Delta+V)u=0
	\end{equation}
	can decay at infinity when $V$ is bounded.
	In the 1960s, Landis conjectured that non-zero solutions cannot decay faster than exponentially. This was disproven by \cite{Meshkov1992} for complex potentials: he found a complex potential and a solution that decay faster than $e^{-C|x|^{4/3}}$ in dimension $2$. For real potentials, the conjecture remains open, but recent work of \cite*{logunov2020landisconjectureexponentialdecay} shows that solutions cannot decay faster than roughly $e^{-CR\log(R)^{1/2}}$ in dimension $2$.
	Their proof involves removing holes in the domain to increase the control of how functions oscillate by reducing the Poincaré constant. Then they use quasi-conformal mappings to transform the problem into one about harmonic functions on a domain with holes.


	A key result from their work is as follows:

	Let $(D_{j})_j$ be a collection of disjoint unit balls with mutual separation at least 100, and let $R'>100$. Then for all harmonic function  $\alpha$ on $B(0,R')\setminus\bigcup_j D_{j}$ such that it does not change sign on $(B(0,R')\cap 5D_{j})\setminus D_{j}$. We get

	\begin{equation}
		\int_{B(0,R')\setminus(B(0,R'/2)\bigcup_j(3D_{j}))} |\alpha(x)|^2dx\geq e^{-C R'\log(R')}\int_{B(0,R'/2)\setminus\bigcup_j(3D_{j}) }|\alpha(x)|^2dx.
	\end{equation}

	To go deeper, \cite*{fernandez-bertolinLandisConjectureSurvey2024} did a complete survey of Landis' conjecture. 

The outline of this article is as follows. In Section~\ref{sect: Preliminaire}, we begin by applying the uniformization theorem to closed Riemann surfaces. This allows us to simplify the problem by working within a ball, thus avoiding the complexities of the original manifold (Section~\ref{subsect: Uniformization}). We then adapt some classical results about the Schr\"odinger operator from \cite*{logunov2020landisconjectureexponentialdecay} to the semiclassical Schr\"odinger operator framework (Section~\ref{subsect: preliminaires}). The preliminaries conclude with a discussion on hole cuttings within the ball framework (Section~\ref{sect: decoupespace}).

Section~\ref{sect: de SchrodAHarmonique} is dedicated to transforming eigenfunctions of the Schr\"odinger operator into harmonic functions. This process unfolds in three steps: (1) reduction to a divergence-free equation (Section~\ref{subsect: eqdivsuresptroue}), (2) application of quasiconformal mappings to reduce the divergence-free equation to a harmonic problem (Section~\ref{subsect: chg de var quasi-conforme}), (3) analysis of the geometric impact of these mappings on our framework (Section~\ref{subsect: distortion de lespace}).

In Section~\ref{sect: Ineg Carleman}, we build upon a result from \cite*{logunov2020landisconjectureexponentialdecay} concerning harmonic functions to derive a statement for Schr\"odinger's eigenfunctions. We examine how quasiconformal mapping reverses geometric transformations, enabling us to transfer estimates back to the Schr\"odinger's eigenfunctions at the expense of space distortions(Section~\ref{subsect: retourarriere}). To control the induced deformations, we employ some analytic tools (Section~\ref{subsect: Gag-Nir-Caccio}).  Finally, to address the holes arising from the transformation of the punctured ball, we use Carleman estimates on shrinking sets. This allow us to bound exponential error terms in the statement on Schödinger's eigenfunctions (Section~\ref{subsect: Carleman_classique}).

The proof is concluded in Section~\ref{subsect: Conclusion}, where we synthesize all preceding results.

\section{Preliminaries}
\label{sect: Preliminaire}
In this paper, we study the eigenfunction $u_h \in L_{loc}^1(M,\R)$ associated with the Schrödinger equation on a closed Riemannian surface $(M,g)$:
\[
(-h^2\Delta_g + V - E)u_h = 0 \quad \text{on } (M,g),
\]
where $V \in L^\infty(M, \R)$ and $E \in I$, with $I$ a compact interval of $\R$. Let $U$ be an open subset of $M$ and  $\mu_g$ the volume form on $(M,g)$.

\textbf{Notation:} To avoid overloading notations, we will use the symbol $\lesssim$ to indicate an order of inferiority to within a universal multiplicative constant factor, and we will use this symbol $\simeq$ to indicate that two quantities are equivalent. 

\subsection{Reduction of the problem}
\label{subsect: Uniformization}
Let $(M,g)$ be a closed Riemannian surface. By the Poincaré uniformization theorem, $M$ is conformally equivalent to a unique closed surface with constant curvature $(M,g_0)$: there is $\phi \in \mathcal{C}^\infty(M)$ positive such that $g=\phi g_0$ (see for instance the book of \cite{de2016uniformization}). Therefore, $u_h$ is solution of
\[
(-h^2\phi^{-1}\Delta_{g_0} + V - E)u_h = 0 \quad \text{on } (M,g_0)
\]
and the measures $\mu_g$ and $\mu_{g_0}$ are equivalent.
The surface $(M,g_0)$ is a quotient of one of the following covering surfaces under the free action of a discrete subgroup of its isometry group:

\begin{itemize}
    \item the Euclidean plane $\R^2$ (0 curvature),
    \item the sphere $\S^2$ (1 curvature),
    \item the hyperbolic disk $\H^2$ (-1 curvature) (see Figure~\ref{fig: Poincare_disk}).
\end{itemize}

\begin{figure}
    \centering
    \includegraphics[width=0.4\linewidth]{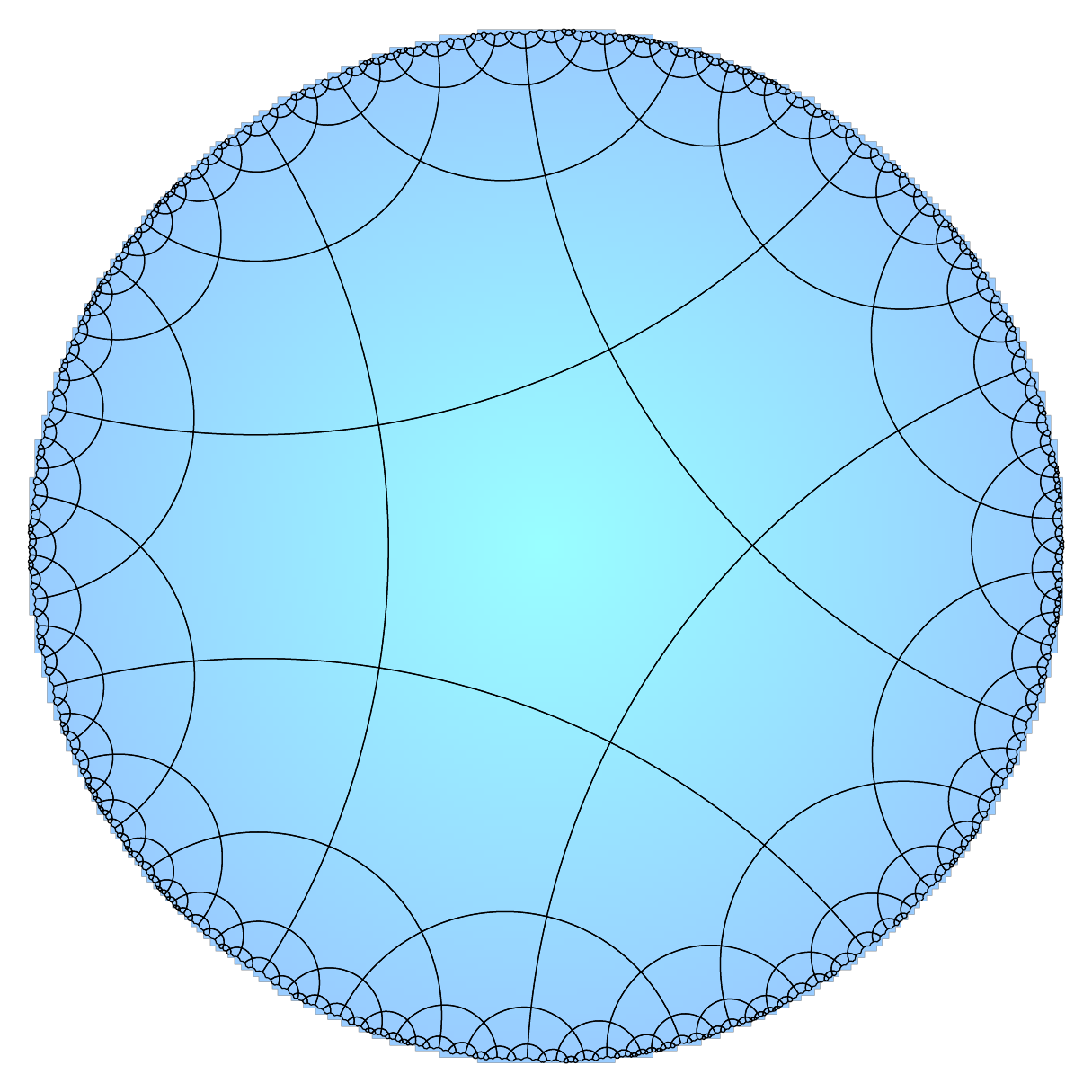}
    \caption{Hyperbolic's disk paved by pentagon (code source \cite{https://doi.org/10.5281/zenodo.7559393})}
    \label{fig: Poincare_disk}
\end{figure}

 Consequently, $(M,g_0)$ can be viewed as a compact quotient $\tilde{M}$ of one of these three surfaces with the inherited metric, and $U$ as an open subset of $\tilde{M}$. These three surfaces can be covered by $\tilde{M}$ via the free action of the discrete subgroup. We can then view the function $u_h$ as a function $u_{per,h}$ on the covering which is periodic by the free action. Moreover, $u_{per,h}$ is a solution of the "periodized" Schrödinger equation on the covering:
\[
(P_{V,per}-E) u_{per,h} := (-h^2\phi_{per}^{-1}\Delta_{g_0} + V_{per} - E)u_{per,h} = 0.
\]

This reduces our analysis to just three cases:

\begin{itemize}
    \item \textbf{Case of the Euclidean plane $\R^2$}: We can replace $U$ with a small ball $B(0, R)$ centered at $0$ inside $U$. It is then sufficient to work in a larger ball $B(0, R_0)$ containing $\tilde{M}$. On this ball, $\mu_{g_0}$ is equivalent to  the Lebesgue measure and $u_{per,h}$ is solution of the equation 
    \[
    (P_{V,per}-E) u_{per,h} := (-h^2\tilde{\phi}\Delta + V_{per} - E)u_{per,h} = 0.
    \]
    with $\tilde{\phi} \in \mathcal{C}^\infty(\overline{B(0,R_0)})$ positive. 

    \item \textbf{Case of the hyperbolic disk $\H^2$}: We can replace $U$ with a small ball $B(0, R)$ centered at $0$ inside $U$. It is then sufficient to work in a larger ball $B(0, R_0)$ within the disk. In fact, the Poincaré disk is equipped with the metric
    \[
    ds^2 = \frac{4\sum_i dx_i^2}{(1 - \sum_i x_i^2)^2},
    \]
    so on $B(0, R_0)$, the metric is equivalent to an Euclidean metric. Consequently, $\mu_{g_0}$ is equivalent to the Lebesgue measure and $\Delta_{g_0}=\frac{(1 - \sum_i x_i^2)^2}{4}\Delta$ with $\Delta$ the Euclidean Laplacian. Thus $u_{per,h}$ is solution of the equation 
    \[
    (P_{V,per}-E) u_{per,h} := (-h^2\tilde{\phi}\Delta + V_{per} - E)u_{per,h} = 0.
    \]
    with $\tilde{\phi} \in \mathcal{C}^\infty(\overline{B(0,R_0)})$ positive. Thus this case can be treated similarly to the Euclidean case.

    \item \textbf{Case of the sphere $\S^2$}: We can replace $U$ with a small ball $B(R)$ centered at the South Pole of $\S^2$ and consider $\tilde{M}$ as $\S^2$. We can then remove a small ball $B(R/2)$ centered on the South Pole inside $B(R)$. The perforated sphere can unfold onto a ball $B(0, R_0)$ with the stereographic projection on the South Pole (see Figure~\ref{fig: Stereographic_projection}). Thus, $B(R)\setminus B(R/2)$ becomes an annulus $A(0, R, R_0)$ at the edge of $B(0, R_0)$. The metric on this projection is
      \[
    ds^2 = \frac{4\sum_i dx_i^2}{(1 + \sum_i x_i^2)^2},
    \]
 so on $B(0, R_0)$, the metric is equivalent to an Euclidean metric. Consequently, $\mu_{g_0}$ is equivalent to the Lebesgue measure and $\Delta_{\tilde{g}_0}=\frac{(1 + \sum_i x_i^2)^2}{4}\Delta$ with $\Delta$ the Euclidean Laplacian. Thus $u_{per,h}$ is solution of the equation 
\[
(P_{V,per}-E) u_{per,h} := (-h^2\tilde{\phi}\Delta + V_{per} - E)u_{per,h} = 0.
\]
with $\tilde{\phi} \in \mathcal{C}^\infty(\overline{B(0,R_0)})$ positive.
    \begin{figure}
        \centering
        \includegraphics[width=0.6\linewidth]{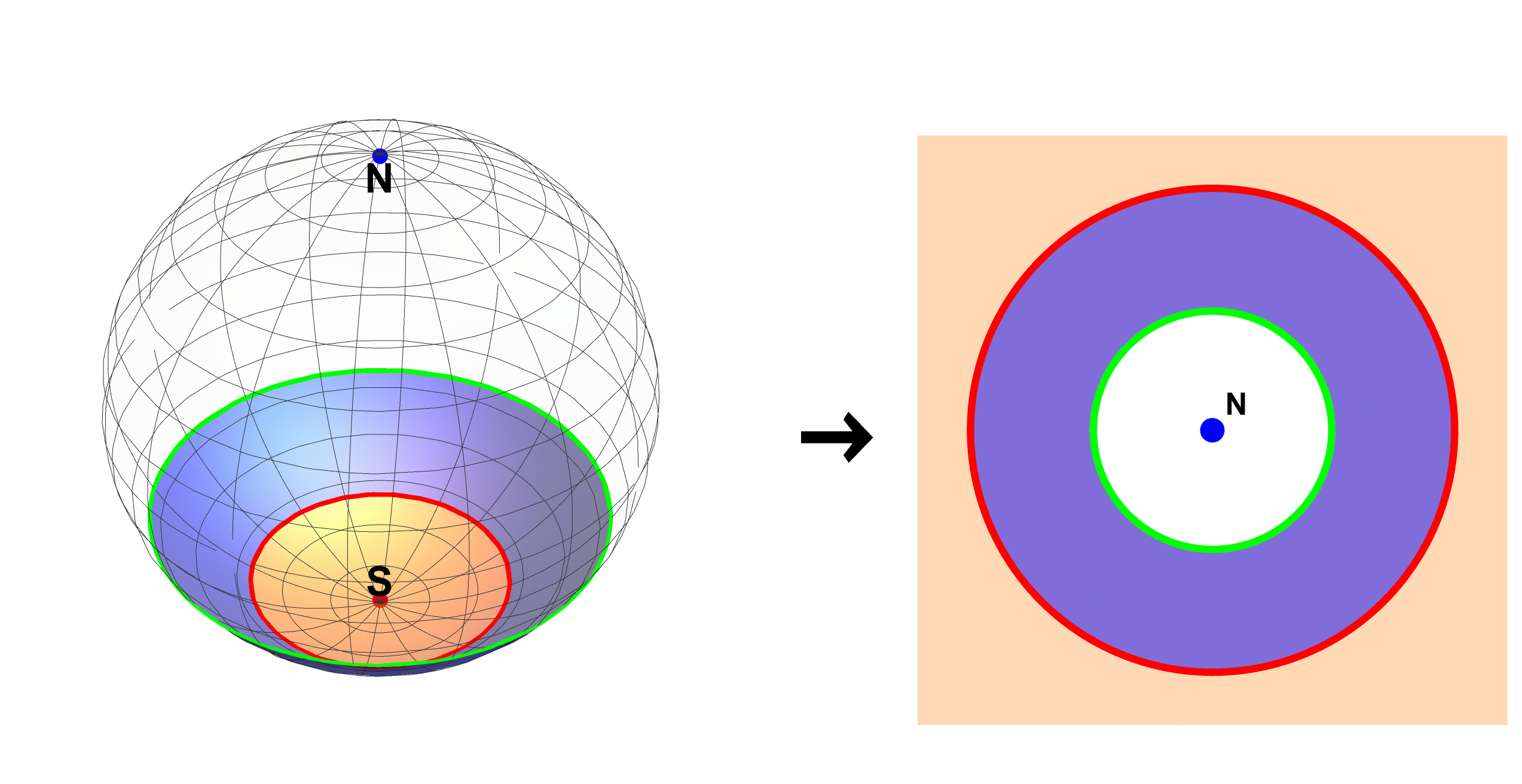}
        \caption{Stereographic projection of $\S^2$ on $\R^2$}
        \label{fig: Stereographic_projection}
    \end{figure}
\end{itemize}
 
 We have reduce these three initial geometric cases to a single unified framework:
a ball $B(0, R_0)$ in a reference space (the plane $\R^2$ after transformation) with the Lebesgue measure. On $B(0, R_0)$, we study  $u_{per,h}$ solution of the equation
\begin{equation}
	\label{eq:equationOrigine}
	(P_{V,per}-E) u_{per,h} := (-h^2\tilde{\phi}\Delta + V_{per} - E)u_{per,h} = 0
\end{equation}
with $\tilde{\phi} \in \mathcal{C}^\infty(\overline{B(0,R_0)})$ positive, $V_{per}\in L^\infty(B(0,R_0),\R)$, $E\in I \subset \R$ compact interval. The function $u_{per,h} \in L_{loc}^1$ is "periodic" and it takes all its values on a set $\tilde{M}$ inside $B(0,R_0)$.

In this framework, two subcases arise for the open set $U$:
\begin{itemize}
    \item $U$ is identified with a ball $B(0, R)$ centered at $0$ and contained within $B(0, R_0)$ (interior case).
    \item $U$ is identified with an annulus $A(0, R, R_0)$ located at the edge of $B(0, R_0)$ (boundary case).
\end{itemize}

 In the equation \ref{eq:equationOrigine}, the function $\tilde{\phi}$ is smooth and bounded by two positive constant on $B(0,R_0)$. Consequently $u_{per,h}$ is solution of $$(-h^2\Delta + \tilde{V})u_{per,h} = 0$$ with $\tilde{V}=\frac{V_{per}-E}{\tilde{\phi}}\in L^\infty$. Thus, for sake of simplicity, we will replace $\tilde{\phi}$ by $1$. This will have no impact on the proof.
\subsection{Some lemmas}
\label{subsect: preliminaires}
Before we proceed, we need to adapt the statements from \cite{logunov2020landisconjectureexponentialdecay} to our semiclassical framework. Complete proofs (if they are not redone) can be found in their article.

\begin{lem}
\label{lem: C1_regularity_solution_Schrodinger}
If $(P_{V,per} -E)u_{per,h} = 0$ on an open set $\Omega_h$ in the sense of distributions and $u_{per,h} \in L_{loc}^1(\Omega_h)$, then $u_{per,h} \in \mathcal{C}^1(\Omega_h)$.
\end{lem}
This lemma is local, so it can be proved with the help of the fundamental solution of $\Delta$ (see \cite{logunov2020landisconjectureexponentialdecay}, Fact $6.5$).

\begin{lem}
\label{lem: small_laplacian_free_solution}
Let $\Omega_h$ be an open subset of $B(0, R_0)$ with Poincaré constant $k_h^2$. Suppose that
\[
h^{-2} k_h^2 \| V_{per} - E \|_{L^\infty} \ll 1.
\]
Then, there exists a weak solution $\varphi_h$ to the equation $(P_{V,per}-E) \varphi_h = 0$ on $\Omega_h$ such that $\varphi_h = 1 + \tilde{\varphi}_h$, where
\[
\tilde{\varphi}_h \in W_0^{1,2}(\Omega_h), \quad \|\tilde{\varphi}_h\|_{L^\infty} \leq C h^{-2} k_h^2 \| V_{per} - E \|_{L^\infty}.
\]
\end{lem}
To prove this, we rewrite the operator in the form $-\Delta + h^{-2}(V_{per} - E)$ and apply Lemma $3.2$ from \cite{logunov2020landisconjectureexponentialdecay}.

\begin{lem}
\label{lem: utility_1}
Let $\Omega_h$ be a bounded open set and let $f_h \in \mathcal{C}^1(\bar{\Omega}_h)$ be a function such that $f_h = 0$ on $\partial \Omega_h$. Then, $f_h \in W_0^{1,2}(\Omega_h)$.
\end{lem}
See Lemma $6.13$ of \cite{logunov2020landisconjectureexponentialdecay}.

\begin{lem}
\label{lem: positivity_radius}
Let $u_h$ be a solution to $(-h^2\Delta + V_{per} - E)u_h = 0$ on a ball $B(x, r_h)$ of $B(0, R_0)$, where $r_h < r_{0,h} = r_0\times h$ with $r_0$ a small constant independent of $h$ and $u_h$. If $u_h$ is continuous and $u_h > 0$ on $\partial B(x, r_h)$, then $u_h > 0$ on $B(x, r_h)$.
\end{lem}

\begin{proof}
We can assume that $u_h$ is greater than a positive constant $\delta_h$ on $\partial B(x, r_h)$. Consider the set $\Omega_h = \{x \in B(x, r_h) : u_h(x) < \delta_h/2\}$. This is an open subset strictly contained in $B(x, r_h)$. If $u_h$ is not positive in $B(x, r_h)$, then $\Omega_h$ is non-empty. Since $u_h \in \mathcal{C}^1(\Omega_h)$ by Lemma~\ref{lem: C1_regularity_solution_Schrodinger} and $u_h = \delta_h/2$ on $\partial \Omega_h$, we deduce from Lemma~\ref{lem: utility_1} that $(u_h - \delta_h/2) \in W_0^{1,2}(\Omega_h)$.

Since $\Omega_h \subset B(x, r_h)$, $\Omega_h$ has a Poincaré constant of $C r_h^2$. For sufficiently small $r_h$, Lemma~\ref{lem: small_laplacian_free_solution} guarantees the existence of $\varphi_h = 1 + \tilde{\varphi}_h$, where $\tilde{\varphi}_h \in W_0^{1,2}(\Omega_h)$ and $\|\tilde{\varphi}_h\|_{L^\infty} \leq 1/2$, such that $-h^2\Delta \varphi_h + (V_{per} - E)\varphi_h = 0$ on $\Omega_h$.

Then, the function $g_h = (\delta_h \varphi_h/2 - u_h)$ belongs to $W_0^{1,2}(\Omega_h)$ and satisfies $-h^2\Delta g_h + (V_{per} - E)g_h = 0$. For any $\beta_h \in \mathcal{C}_0^\infty(\Omega_h)$,
\[
h^2 \int_{\Omega_h} \nabla g_h \cdot \nabla \beta_h = - \int_{\Omega_h} (V_{per} - E) g_h \beta_h.
\]
Taking the limit $\beta_h \rightarrow g_h$ in $W_0^{1,2}(\Omega_h)$, we obtain:
\[
h^2 \int_{\Omega_h} |\nabla g_h|^2 = - \int_{\Omega_h} (V_{per}-E) g_h^2 \leq \|V_{per} - E\|_{L^\infty} \int_{\Omega_h} g_h^2 \leq \|V_{per} - E\|_{L^\infty} C r_h^2 \int_{\Omega_h} |\nabla g_h|^2.
\]
For $r_h < h (C \|V_{per} - E\|_{L^\infty})^{-1/2}$, this inequality holds only if $g_h = 0$. Thus, $u_h = \delta_h \varphi_h/2$ on $\Omega_h$. Since $\|\tilde{\varphi}_h\|_{L^\infty} \leq 1/2$, we have $u_h > \delta_h/4$ on $\Omega_h$, and hence on $B(x, r_h)$.
\end{proof}

\subsection{Perforation of space}
\label{sect: decoupespace}
Let's finish the preliminaries by preparing the ground for the proof. To do this, we will slightly modify the ball $B(0,R_0)$. 
\begin{defn}
Let $F_{0,h}$ denote the zero-set of $u_{per,h}$ on $B(0, R_0)$.
\end{defn}

Let $\epsilon(h) > 0$ and $C > 2$ be two parameters to be determined later. We begin by filling $B(0, R_0) \setminus F_{0,h}$ with a maximal collection of disjoint open balls of radius $(C+1)h\epsilon(h)$. By maximality, every point in $B(0, R_0) \setminus F_{0,h}$ lies within a distance of $2(C+1)h\epsilon(h)$ from the center of one of these balls, from $F_{0,h}$, or from $\partial B(0, R_0)$. Otherwise, we could add another ball centered at that point.

\begin{defn}
We define $F_{1,h}$ as the collection of balls of radius $h\epsilon(h)$, centered at the same points as the $(C+1)h\epsilon(h)$-radius balls constructed above (see Figure~\ref{fig:esptroue}).
\end{defn}

Thus, any ball in $F_{1,h}$ is separated by a distance greater than $Ch\epsilon(h)$ from any other ball in $F_{1,h}$, from $\partial B(0, R_0)$, and from $F_{0,h}$. The set $F_{0,h} \cup F_{1,h}$ is $3Ch\epsilon(h)$-dense in $B(0, R_0)$. Thanks to this construction, we can control the Poincaré constant of the domain $\Omega_h := B(0, R_0) \setminus (F_{0,h} \cup F_{1,h})$:

\begin{figure}
    \centering
    \includegraphics[width=0.6\linewidth]{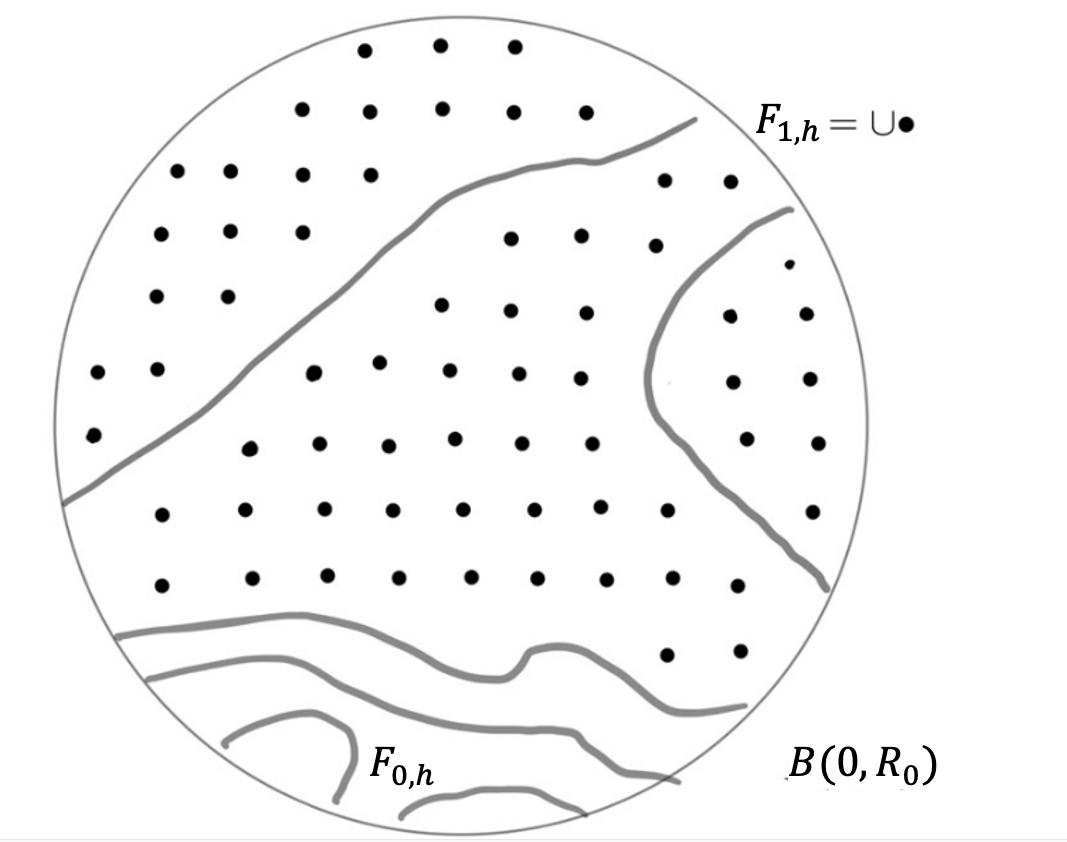}
    \caption{Puncturing outside the nodal domains in $B(0, R_0)$ (Figure from \cite{logunov2020landisconjectureexponentialdecay}).}
    \label{fig:esptroue}
\end{figure}

\begin{lem}
For $\epsilon(h)$ sufficiently small, the Poincaré constant of $\Omega_h$ is bounded by $C'h^2\epsilon(h)^2$, where $C'$ is a constant depending only on $C$.
\end{lem}

\begin{proof}
See Appendix~\ref{subsect: small_PC}.
\end{proof}
	
\section{From Schrödinger to Harmonic equation}
\label{sect: de SchrodAHarmonique}
We recall that $u_{per,h}$ is a solution of
$$
(P_{V,per}-E) u_{per,h} := (-h^2\Delta + V_{per} - E)u_{per,h} = 0,
$$
where $V_{per}$ is a potential in $L^\infty(B(0,R_0), \R)$ and $E \in I$ is a compact interval of $\R$.
The aim of this section is to reduce the study of Schrödinger eigenfunctions to that of harmonic functions. First, we reduce the problem to a divergence-free equation, and then we use the theory of quasi-conformal mappings to transform it into a harmonic function problem.

\subsection{Divergence type equation on the perforated space}
\label{subsect: eqdivsuresptroue}
We recall that by Lemma~\ref{lem: small_laplacian_free_solution}, if $\epsilon(h)$ is sufficiently small (and consequently, the Poincaré constant of $\Omega_h$ is also small), then there exists a weak solution $\varphi_h$ to the equation $(-h^2\Delta + V_{per} - E)\varphi_h = 0$ on $\Omega_h$ such that $\varphi_h = 1 + \tilde{\varphi}_h$, where $\tilde{\varphi}_h \in W_0^{1,2}(\Omega_h)$ and $\|\tilde{\varphi}_h\|_{L^\infty} = \mathcal{O}(\epsilon(h)^2)$.

We extend $\varphi_h$ by $1$ outside $\Omega_h$ and consider the function
\begin{equation}
f_h = \frac{u_{per,h}}{\varphi_h}.
\end{equation}

\begin{lem}
The function $f_h \in W_{loc}^{1,2}(B(0, R_0))$ is a weak solution to the equation
\begin{equation}
\mbox{div}(\varphi_h^2 \nabla f_h) = 0
\end{equation}
on $\tilde{\Omega}_h := B(0, R_0) \setminus F_{1,h}$.
\end{lem}

The proof of this lemma can be found in \cite{logunov2020landisconjectureexponentialdecay}, Lemma~$4.1$.

\subsection{Quasi-conformal mapping}
\label{subsect: chg de var quasi-conforme}
In $\tilde{\Omega}_h$, we can show, using the Poincaré lemma, that there exists a function $\tilde{f}_h \in W_{loc}^{1,2}$ defined locally such that the local form $\omega_h = f_h + i\tilde{f}_h$ satisfies the Beltrami equation:
\begin{equation}
\frac{\partial \omega_h}{\partial \bar{z}} = \mu_h \frac{\partial \omega_h}{\partial z},
\label{eq: Beltrami}
\end{equation}
with coefficient:
\begin{defn}
\label{def: definition du coeff de Beltrami}
\begin{equation}
\mu_h := \frac{1 - \varphi_h^2}{1 + \varphi_h^2} \cdot \frac{\partial_x f_h + i \partial_y f_h}{\partial_x f_h - i \partial_y f_h}.
\end{equation}
\end{defn}

Unlike $\omega_h$, $\mu_h$ is globally defined on $\tilde{\Omega}_h$ and satisfies:
\begin{equation}
|\mu_h| \leq \left| \frac{1 - \varphi_h^2}{1 + \varphi_h^2} \right| \lesssim \epsilon(h)^2.
\end{equation}
We extend $\mu_h$ by $0$ on the whole complex plane. Thus, $\mu_h$ is a measurable function compactly supported with $|\mu_h| \lesssim \epsilon(h)^2 < 1$. 

By invoking an existence theorem due to Ahlfors (\cite{ahlfors2006lectures}, Chapter~$5$, Section~$B$, Theorem~$1$), we obtain the following result:
\begin{prop}
\label{prop: def_de_psi_h}
There exists a $K_h$-quasi-conformal homeomorphism $\psi_h$ of the plane (where the dilatation coefficient $K_h$ is given by $K_h = \frac{1 + |\mu_h|}{1 - |\mu_h|}$) satisfying the following properties:
\begin{itemize}
    \item $\psi_h \in W_{loc}^{1,2}$,
    \item $\frac{\partial \psi_h}{\partial \bar{z}} = \mu_h \frac{\partial \psi_h}{\partial z}$,
    \item $K_h \leq 1 + \tilde{C} \epsilon(h)^2$, where $\tilde{C} > 0$,
    \item $\psi_h(0) = 0$ and $\psi_h(\infty) = \infty$.
\end{itemize}

Moreover, the following estimate holds for the derivative of $\psi_h$:
\begin{equation}
\|\partial_z \psi_h - 1\|_p \leq \frac{C_p \|\mu_h\|_p}{1 - C_p \|\mu_h\|_\infty},
\end{equation}
where the constant $C_p$ is defined by
\begin{equation}
C_p := \frac{\pi^2}{4} \cdot \frac{1}{\left(\frac{p}{p-1}\right)^{2/p} - 1} \lesssim_{p \to +\infty} p^2.
\end{equation}
This estimate is valid when $C_p \|\mu_h\|_\infty < 1$, which holds for large $p$ since $\epsilon(h)$ is small enough.
\end{prop}

By Stoilow's factorization theorem (see \cite{Astala2008-dh}), there exists a holomorphic function $\Gamma_h$ such that
\begin{equation}
\omega_h = \Gamma_h\circ \psi_h \quad \text{locally on } \tilde{\Omega}_h.
\end{equation}
Thus, $f_h \circ \psi_h^{-1}$ is harmonic on $\psi_h(\tilde{\Omega}_h)$. By the Riemann uniformization theorem, there exists a conformal biholomorphic homeomorphism mapping $\psi_h(B(0, R_0))$ to $B(0, R_0)$ and $\psi_h(0)$ to $0$. The composition of these two homeomorphisms is a $K_h$-quasi-conformal mapping $g_h$ satisfying the same Beltrami equation~\eqref{eq: Beltrami} as $\omega_h$ and fixing $0$. Therefore, $f_h \circ g_h^{-1}$ is harmonic on $g_h(\tilde{\Omega}_h)\subset B(0,R_0)$.

\subsection{Space distortion}
\label{subsect: distortion de lespace}
By Mori's theorem (see \cite{ahlfors2006lectures}, Chapter~$3$, Section~$C$), we can control the distortion of distances by $g_h$ inside $B(0, R_0)$:
\begin{equation}
\frac{1}{16} \left| \frac{z_1 - z_2}{R_0} \right|^{K_h} \leq \frac{|g_h(z_1) - g_h(z_2)|}{R_0} \leq 16 \left| \frac{z_1 - z_2}{R_0} \right|^{1/K_h}.
\label{eq: InegdeMori}
\end{equation}

We define
\begin{defn}
\label{def: epsilon(h)}
\begin{equation}
\epsilon(h) = \frac{c}{\log(R_0 / h)},
\end{equation}
where $c > 0$ is small.
\end{defn}

Thus, $K_h \in [1, 1 + \tilde{C}c^2 / (\log(R_0 / h))^2[$, and $(R_0)^{K_h} \sim R_0 \sim (R_0)^{1/K_h}$. For sufficiently small $h$, if $h\epsilon(h) \leq |z_1 - z_2| \leq 2R_0$, then
\begin{equation}
\frac{1}{32} |z_1 - z_2| \leq |g_h(z_1) - g_h(z_2)| \leq 32 |z_1 - z_2|.
\label{eq: inegsurdistance}
\end{equation}

Thus, the function $\tilde{f}_h := f_h \circ g_h^{-1}$ is harmonic on $B(0, R_0) \setminus g_h(F_{1,h})$, where $g_h(F_{1,h})$ is a disjoint union of sets with diameters proportional to $h\epsilon(h)$. Each set in $g_h(F_{1,h})$ is contained in a ball of radius $32h\epsilon(h)$. Let us denote these balls by $D_{j,h}$. Each set in $g_h(F_{1,h})$ is separated by at least $\frac{C}{32}h\epsilon(h)$ from the other sets and from the zero set of $\tilde{f}_h$. Therefore, the balls $D_{j,h}$ are separated from each other and from $\tilde{f}_h^{-1}(\{0\})$ by a distance greater than $\frac{C}{32}h\epsilon(h) - 128h\epsilon(h) =: 32C_1h\epsilon(h)$. Moreover, $\tilde{f}_h$ does not change sign on $32C_1D_{j,h} \setminus D_{j,h}$. We have:
\begin{equation}
C_1 = \frac{C}{32} - 4 > 100,
\end{equation}
provided that $C$ is sufficiently large.
\section{From control on Harmonic functions to one on Scrödinger eigenfunctions}
\label{sect: Ineg Carleman}
We recall that
$\tilde{f}_h = \frac{u_{per,h}}{\varphi_h} \circ g_h^{-1}$
is harmonic on $\tilde{\Omega}_h = B(0, R_0) \setminus \bigcup_{j} D_{j,h}$, where $D_{j,h}$ are balls with radius $32h\epsilon(h)$ and are separated from each other and from $\tilde{f}_h^{-1}(\{0\})$ by a distance of $32C_1 h\epsilon(h)$ with $C_1 \gg 100$ (see Section~\ref{subsect: distortion de lespace}). We also recall that $|\varphi_h - 1|$ is bounded almost everywhere by a constant smaller than $\tilde{C}\epsilon(h)^2$ with $\tilde{C} > 0$, and $g_h = R_h \circ \psi_h$ is a homeomorphism of $B(0, R_0)$ with $\psi_h$ a quasi-conformal homeomorphism defined in Proposition~\ref{prop: def_de_psi_h} and $R_h$ a bi-holomorphic mapping between $\psi_h(B(0, R_0))$ and $B(0, R_0)$.

In Section~\ref{subsect: Uniformization}, we concluded that there are two cases to consider:
\begin{itemize}
    \item \textbf{Interior case:} Let $B(0, R)$ be a small ball with radius $R$ centered at $0$ such that $B(0, R) \subset \tilde{M} \subset B(0, R_0)$, where $\tilde{M}$ is a compact set in $B(0, R_0)$ that represents the Riemann surface $M$.
    \item \textbf{Boundary case:} Let $ A(0, R, R_0)$ be a ring centered at $0$ with $0 < R < R_0$.
\end{itemize}

\subsection{A few steps backwards}
\label{subsect: retourarriere}
Based on the demonstration of the "Toy problem" theorem in \cite{logunov2020landisconjectureexponentialdecay} (Section $6.1$), we prove:
\begin{thm}
Let $((D_{j,h})_j)_h$ be a family of ball collections on the plane, $32C_1 h\epsilon(h)$ separated (with $C_1 \gg 100$), with radius $32h\epsilon(h)$. Let $\tilde{R}_0 \gg 1$.
\begin{itemize}
    \item \textbf{Interior case:}\\
    Let $0 < \tilde{R} < \tilde{R}_0$. There exist constants $\tilde{C}' > C' > 0$ such that for $h$ small enough, for any family $(\tilde{f}_h)_h$ of harmonic functions $B(0,\tilde{R}_0)\setminus \bigcup_j D_{j,h}$  such that $\tilde{f}_h$ does not change sign in each $B(0, \tilde{R}_0) \cap 5D_{j,h} \setminus D_{j,h}$:

		\begin{multline}
			\int_{B(0,\tilde{R})\setminus \bigcup_j(3D_{j,h})} \tilde{f}_h^2+e^{-\tilde{C}'\frac{\log \left((h\epsilon(h))^{-1}\right)}{h\epsilon(h)}}\int_{B(0,\tilde{R}_0)\setminus B(0,\frac{1}{2}\tilde{R}_0)\bigcup_j (3D_{j,h})} \tilde{f}_h^2 \\
			\geq e^{-C'\frac{log\left((h\epsilon(h))^{-1}\right)}{h\epsilon(h)}}\int_{B(0,\frac{1}{2}\tilde{R}_0)\setminus B(0,\tilde{R})\bigcup_j(3D_{j,h})} \tilde{f}_h^2 .
		\end{multline}

    \item \textbf{Boundary case:}\\
    Let $0 < \tilde{R}_1 < \tilde{R}_2 < \tilde{R}_0$. There exists a constant $C' > 0$ such that for $h$ small enough, for any family $(\tilde{f}_h)_h$ of harmonic functions on $B(0,\tilde{R}_0)\setminus \bigcup_j D_{j,h}$ such that $\tilde{f}_h$ does not change sign in each $B(0, \tilde{R}_0) \cap 5D_{j,h} \setminus D_{j,h}$:

		\begin{equation}
			\int_{B(0,\tilde{R}_2)\setminus B(0,\tilde{R}_1)\bigcup_j(3D_{j,h})} \tilde{f}_h^2\geq e^{-C'\frac{\log \left((h\epsilon(h))^{-1}\right)}{h\epsilon(h)}}\int_{B(0,\tilde{R}_1)\setminus \bigcup_j(3D_{j,h})} \tilde{f}_h^2. 
		\end{equation}
\end{itemize}
\end{thm}

\begin{proof}
See Appendix~\ref{subsect: Toy problem}.
\end{proof}

For the interior case, consider $B(0, \tilde{R}) \subset g_h(B(0, R'))$ with $R'<R$ and replace $\tilde{R}_0$ by $R_0 - \tilde{\delta}$ (with $\tilde{\delta} > 0$ small) in the theorem such that $B(0, \tilde{R}_0) \subset g_h(B(0, R'_0 = R_0 - \delta))$ (with $\delta > 0$ small). For the boundary case, consider $B(0, \tilde{R}_2) \setminus B(0, \tilde{R}_1) \subset g_h(A(0, R_1, R_2))$ with $R<R_1 < R_2 < R_0$.

If we go backward by $g_h$, we get:
\begin{multline}
\int_{g_h^{-1}(B(0, \tilde{R}) \setminus \bigcup_j (3D_{j,h}))} |f_h|^2 |j(g_h)| + e^{-\tilde{C}' \frac{\log((h\epsilon(h))^{-1})}{h\epsilon(h)}} \int_{g_h^{-1}(B(0, \tilde{R}_0) \setminus B(0, \frac{1}{2}\tilde{R}_0) \bigcup_j (3D_{j,h}))} |f_h|^2 |j(g_h)|\\ \geq e^{-C' \frac{\log((h\epsilon(h))^{-1})}{h\epsilon(h)}} \int_{g_h^{-1}(B(0, \frac{1}{2}\tilde{R}_0) \setminus B(0, \tilde{R}) \bigcup_j (3D_{j,h}))} |f_h|^2 |j(g_h)|,
\end{multline}
and
\begin{equation}
\int_{g_h^{-1}(B(0, \tilde{R}_2) \setminus B(0, \tilde{R}_1) \bigcup_j (3D_{j,h}))} |f_h|^2 |j(g_h)| \geq e^{-C' \frac{\log((h\epsilon(h))^{-1})}{h\epsilon(h)}} \int_{g_h^{-1}(B(0, \tilde{R}_1) \setminus \bigcup_j (3D_{j,h}))} |f_h|^2 |j(g_h)|,
\end{equation}
where $j(g_h)$ is the Jacobian of $g_h$. Since $g_h$ maps $B(0, R_0)$ onto $B(0, R_0)$, so does $g_h^{-1}$. Moreover, $g_h^{-1}$ is also a $K_h$-quasi-conformal homeomorphism like $g_h$ (see Section~\ref{subsect: chg de var quasi-conforme}). Thanks to the distance distortion \eqref{eq: inegsurdistance}, we know that each ball $3D_{j,h}$ is mapped to a set contained in a ball $D'_{j,h}$ with radius $3 \times 32^2 h\epsilon(h)$. These balls are separated by $(\frac{C_1}{32} h\epsilon(h) - 4 \times 3 \times 32^2 h\epsilon(h)) = \tilde{C}_1 h\epsilon(h) > 100\ h\epsilon(h)$ if $C_1$ is large enough (which means choosing $C$ large enough). Finally, $g_h^{-1}(B(0, \frac{1}{2}\tilde{R}_0))$ contains $B_{\tilde{M}}$, a ball that contains $\tilde{M}$, a compact set which represents $M$ in $B(0, R_0)$.

This leads to the following inequalities:

\begin{itemize}
	\item \textbf{Interior case:}\\
	\begin{multline}
		\int_{B(0,R')} |f_{h}|^2|j(g_h)|+e^{-\tilde{C}' \frac{log((h\epsilon(h))^{-1})}{h\epsilon(h)}}\int_{B(0,R'_0)\setminus B(0,R')}|f_h|^2|j(g_h)| \\
		\geq e^{-C' \frac{log((h\epsilon(h))^{-1})}{h\epsilon(h)}}\int_{B_{ \tilde{M}}\setminus B(0,R')\bigcup_j D'_{j,h} }|f_h|^2|j(g_h)|,
	\end{multline}
	\item \textbf{Boundary case:}\\
	\begin{equation}
		\int_{A(0,R_1,R_2)} |f_{h}|^2|j(g_h)|\geq e^{-C'\frac{log((h\epsilon(h))^{-1})}{h\epsilon(h)}}\int_{B(0,R_1)\setminus \bigcup_j D'_{j,h}} |f_{h}|^2|j(g_h)|.
	\end{equation}
\end{itemize}
	
Because $| \varphi_h - 1 |$ is bounded by $1/2$ if $\epsilon(h)$  is small enough (lemma:~\ref{lem: small_laplacian_free_solution}), we can replace $f_h$ by $u_{per,h}$, where $u_{per,h}$ is a solution to $(-h^2\Delta + V_{per} - E)u_{per,h} = 0$:

\begin{itemize}
	\item \textbf{Interior case:}\\
		\begin{multline}
		\label{eq: ineg carleman avec j(g)}
		e^{-C' \frac{log((h\epsilon(h))^{-1})}{h\epsilon(h)}}\int_{B_{ \tilde{M}}\setminus B(0,R')\bigcup_j D'_{j,h} }|u_{per,h}|^2|j(g_h)|
		\\ \lesssim \int_{B(0,R')} |u_{per,h}|^2|j(g_h)|+e^{-\tilde{C}' \frac{log((h\epsilon(h))^{-1})}{h\epsilon(h)}}\int_{B(0,R'_0)\setminus B(0,R')}|u_{per,h}|^2|j(g_h)|, 
	\end{multline}
	\item \textbf{Boundary case:}\\
	\begin{equation}
		e^{-C'\frac{log((h\epsilon(h))^{-1})}{h\epsilon(h)}}\int_{B(0,R_1)\setminus \bigcup_j D'_{j,h}} |u_{per,h}|^2|j(g_h)|\lesssim \int_{A(0,R_1,R_2)} |u_{per,h}|^2|j(g_h)|.
	\end{equation}
\end{itemize}
	
\subsection{Distortions control}
\label{subsect: Gag-Nir-Caccio}

We now seek to control the term $j(g_h)$. This can be rewritten in complex form:
\begin{equation}
		\label{eq: jacobienComlexe}
		j(g_h)=|\partial_z g_h|^2-|\partial_{\Bar{z}}g_h|^2=(1-|\mu_h|^2)|\partial_z g_h|^2
\end{equation}
When $\epsilon(h)$ is small enough, $|\mu_h| < 1/2$, so we need only to focus on $|\partial_z g_h|^2$. Furthermore:
\begin{equation}
		\label{eq: derivee_de_gh}
		\partial_z g_h = \partial_z \psi_h \times (\partial_z R_h) \circ \psi_h
\end{equation}
By Proposition~\ref{prop: def_de_psi_h}, we have some controls on $\partial_z \psi_h$. But first, we need to get rid of $\partial_z R_h$.
\begin{lem}
Let $K$ be a compact set of $B(0, R_0)$ such that $d(K, B(0, R_0)^c) > \delta > 0$. For $h$ small enough:
\begin{itemize}
    \item There is a constant $\tilde{A} > 0$ independent of $h$ such that:
    \begin{equation}
    \sup_{z \in \psi_h(K)} |\partial_z R_h(z)| \leq \tilde{A}.
   \end{equation}

    \item There is a constant $\tilde{B} > 0$ independent of $h$ such that:
    \begin{equation}
    \frac{1}{\tilde{B}} \leq \inf_{z \in \psi_h(K)} |\partial_z R_h(z)|.
    \end{equation}
\end{itemize}
\end{lem}
\begin{proof}
In Appendix~\ref{subsect: controle de partialz Rh(z)}.
\end{proof}

We deduce from the lemma:
\begin{itemize}
\item \textbf{Interior case}:
	\begin{multline}
		e^{-C' \frac{log((h\epsilon(h))^{-1})}{h\epsilon(h)}}\int_{B_{ \tilde{M}}\setminus B(0,R')\bigcup_j D'_{j,h}} \left| u_{per,h}(z)\right|^2 \left| \partial_z \psi_h\left(z\right)\right|^2 dz \\   
		\lesssim \int_{B(0,R')} \left|u_{per,h}(z)\right|^2\left|\partial_z \psi_h\left(z\right)\right|^2dz+e^{-\tilde{C}' \frac{log((h\epsilon(h))^{-1})}{h\epsilon(h)}}\int_{B(0,R_0')\setminus B(0,R')}|u_{per,h}(z)|^2\left|\partial_z \psi_h\right|^2dz,
	\end{multline}
\item \textbf{Boundary case}:
	\begin{equation}
		e^{-C' \frac{log((h\epsilon(h))^{-1})}{h\epsilon(h)}}\int_{B(0,R_1)\setminus \bigcup_j D'_{j,h}} \left| u_{per,h}(z)\right|^2 \left| \partial_z \psi_h\left(z\right)\right|^2 dz  
		\lesssim \int_{A(0,R_1,R_2)} \left|u_{per,h}(z)\right|^2\left|\partial_z \psi_h\left(z\right)\right|^2dz.
	\end{equation}
\end{itemize}

It remains to get rid of $|\partial_z \psi_h(z)|$. We perform this little calculation:
	\begin{align}
	\left|\partial_z \psi_h \left(z\right)\right|^2
		&\geq \left|\partial_z \psi_h \left(z\right)+1-1\right|^2\nn\\
		&\geq \left(\left|\partial_z \psi_h \left(z\right)-1\right|^2+1-2\left(\left|\partial_z \psi_h \left(z\right)-1\right|\right)\right)\nn\\
		&\geq - \left|\partial_z \psi_h \left(z\right)-1\right|^2 + \frac{1}{2} 
	\end{align}
we use here the fact that $2ab \leq \frac{1}{2}a^2 + 2b^2$ with $a = 1$ and $b = |\partial_z \psi_h(z) - 1|^2$. With this calculation, we obtain:
\begin{itemize}
\item \textbf{Interior case}:
	\begin{multline}
		e^{-C'  \frac{log((h\epsilon(h))^{-1})}{h\epsilon(h)}}\left(\int_{B_{\tilde{M}}\setminus B(0,R')\bigcup_j D'_{j,h}}  \frac{1}{2} |u_{per,h}(z)|^2 dz -\int_{B_{\tilde{M}}} |u_{per,h}(z)|^2\left|\partial_z \psi_h \left(z\right)-1\right|^2dz\right) \\   
		\lesssim \int_{B(0,R')} \left|u_{per,h}(z)\right|^2\left|\partial_z \psi_h\left(z\right)\right|^2dz+e^{-\tilde{C}'  \frac{log((h\epsilon(h))^{-1})}{h\epsilon(h)}}\int_{B(0,R_0')\setminus B(0,R')} \left|u_{per,h}(z)\right|^2\left|\partial_z \psi_h\left(z\right)\right|^2dz
	\end{multline}
\item \textbf{Boundary case}:
	\begin{multline}
		e^{-C'  \frac{log((h\epsilon(h))^{-1})}{h\epsilon(h)}}\left(\int_{B(0,R_1)\setminus \bigcup_j D'_{j,h}} \frac{1}{2} |u_{per,h}(z)|^2 dz -\int_{B(0,R_1)} |u_{per,h}(z)|^2\left|\partial_z \psi_h \left(z\right)-1\right|^2dz \right) \\
		\lesssim \int_{A(0,R_1,R_2)} \left|u_{per,h}(z)\right|^2\left|\partial_z \psi_h\left(z\right)\right|^2dz.
	\end{multline}
\end{itemize}

 By Hölder's inequality, with the couple $(p, p') \in \mathbb{R}^2$ to be fixed later, we get:
	\begin{align}
		 \int_{K_1} \left|u_{per,h}\left(z\right)\right|^2\left|\partial_z \psi_h\left(z\right)\right|^2dz
		\leq \left(\int_{K_1} |u_{per,h}(z)|^{2p'}dz\right)^{\frac{2}{2p'}}\left(\int_{K_1} \left|	\partial_z \psi_h \left(z\right)\right|^{2p}dz\right)^{\frac{2}{2p}}
	\end{align}
with $K_1$ a compact set of $B(0,R_0)$.

	\begin{align}
		\label{eq: Holder sur A1}
		\int_{K_2} |u_{per,h}(z)|^2\left|\partial_z \psi_h \left(z\right)-1\right|^2dz
		\leq \left(\int_{K_2} |u_{per,h}(z)|^{2p'}dz\right)^{\frac{2}{2p'}}\left(\int_{K_2} \left|\partial_z \psi_h \left(z\right)-1\right|^{2p}dz \right)^{\frac{2}{2p}}
	\end{align}
with $K_2$  compact set of $B(0,R_0)$.

From Proposition~\ref{prop: def_de_psi_h}, we know that:
	\begin{equation}
		\label{eq: Ahlfors}
		\lVert \partial_z \psi_h-1 \lVert_{L^{2p}(B(0,R_0))}\leq \frac{|B(0,R_0)|^{1/2p}C_{2p} \lVert \mu_h \lVert_\infty}{1-C_{2p} \lVert \mu_h \lVert_\infty} 
	\end{equation}
with:
	\begin{equation}
		\label{eq:cstAhlfors}
		C_p:=\frac{\pi^2}{4}\frac{1}{\frac{p}{p-1}^{2/p}-1}\sim_{p \rightarrow +\infty}C\,p^2
	\end{equation}
	
as soon as:
	
	\begin{equation}
		\label{eq: condition_dutilisation_controle_mu}
		C_{2p} \lVert \mu_h \lVert_\infty <1.
	\end{equation}

To reduce to an $L^2$ norm on $u_{per,h}$, we do an interpolation:
\begin{prop}[Gagliardo-Nirenberg \cite{ASNSP_1959_3_13_2_115_0}, Lecture 2, on $B(0, R_0)$]
\label{prop: Gag-Nir}
Let $1 \leq q, r \leq +\infty$ be two real numbers, $j$ and $m$ two non-negative integers, and $\theta \in [j/m, 1]$ such that:
		\begin{equation}
			\frac{1}{2p'}= \frac{j}{2}+\theta \left(\frac{1}{r}-\frac{m}{2} \right)+\frac{1-\theta}{q}
		\end{equation}
Then we get for $u \in L_{loc}^1(B(0, R_0))$:
		\begin{equation}
			\lVert D^j u\lVert_{L^{2p'}(B(0,R_0))}\leq C(p')\lVert u\lVert_{W^{m,r}(B(0,R_0))}^{\theta}\lVert u\lVert_{L^q(B(0,R_0))}^{1-\theta} 
		\end{equation}
with $C(p')$ bounded near $1$.
\end{prop}
	
By taking $q=2$, $j=0$, $r=2$, $m=1$ and $1-\theta=1/p'$, and $u=u_{per,h}\,\chi$ with $\chi$ a test function with support in $B(0,R_0)$ such that $\chi=1$ on $K_1$ or $K_2$,  we obtain:
		\begin{equation}
		\lVert u_{per,h}\lVert_{L^{2p'}(K)}\leq C(p')\lVert u_{per,h}\lVert_{H^1(K')}^{1/p}\lVert u_{per,h}\lVert_{L^2(K_2)}^{1/p'} 
	\end{equation}
with $K = K_1$ or $K_2$, and $K'$ is a set which is containing $K$.

So now we have $L^2$ and $H^1$ norms. In order to replace the $H^1$ norm by an $L^2$ norm, we use a Cacciopoli lemma:
\begin{prop}[Cacciopoli]
\label{prop: Caccio}
Let $\eta \in \mathcal{C}_0^\infty(\Omega)$ with $\Omega$ an open set of $\mathbb{R}^2$. We get:
		\begin{equation}
			\int_{\Omega}|\eta\nabla u_{per,h}|^2 \leq 6\int_{\Omega}u_{per,h}^2 \left( |\nabla \eta|^2 + \frac{|V_{per}-E|}{h^2}\eta^2 \right)
		\end{equation}
\end{prop}

\begin{proof}
In Appendix, Section~\ref{subsect: Cacciopoli}.
\end{proof}

By taking $\eta = 1$ on $K'$ and $\text{Supp}(\eta) \subset K''$ with $K' \subset K''$, we therefore deduce from Cacciopoli (Proposition~\ref{prop: Caccio}):
	\begin{equation}
		\lVert u_{per,h}\lVert_{H^1(K')}\lesssim h^{-1}\lVert u_{per,h}\lVert_{L^2(K'')}
	\end{equation}
Consequently we get this interpolation estimate:
	\begin{equation}
		\label{eq: interpolation norme L2p}
		\lVert u_{per,h}\lVert_{L^{2p'}(K)}\lesssim h^{-1/p}\lVert u_{per,h}\lVert_{L^2{(K'')}}.
	\end{equation}
	We choose $p = \log(h^{-1})$, so that $h^{-1/p}=1$.
	
Thus with $K_2$, we get from Equations~\eqref{eq: interpolation norme L2p},~\eqref{eq:cstAhlfors},~\eqref{eq: Ahlfors}, and~\eqref{eq: Holder sur A1}:
	\begin{equation}
		 \int_{K_2} \left|u_{per,h}\left(z\right)\right|^2\left|\partial_z \psi_h\left(z\right)\right|^2dz \lesssim \log(h^{-1})^4 \lVert \mu_h\lVert_\infty^2 \|u_{per,h}\|_{L^2(K_2")}^2
	\end{equation}
And according to the definition of $\mu_h$ (see Definition~\ref{def: definition du coeff de Beltrami}), we obtain this estimate:
\begin{equation}
	\int_{K_2} \left|u_{per,h}\left(z\right)\right|^2\left|\partial_z \psi_h\left(z\right)\right|^2dz \lesssim (\log(h^{-1})\epsilon(h))^4 \|u_{per,h}\|_{L^2(K_2")}^2.
\end{equation}
This makes sense because the condition~\eqref{eq: condition_dutilisation_controle_mu} is verified in this case. Indeed, $\epsilon(h) = \frac{c}{\log(R_0/h)}$ (see Definition~\ref{def: epsilon(h)}), thus $\log(h^{-1})\epsilon(h) \ll 1$ for $c$ small enough. Hence:
\begin{equation}
	\int_{K_2} \left|u_{per,h}\left(z\right)\right|^2\left|\partial_z \psi_h\left(z\right)\right|^2dz=o_c(1) \|u_{per,h}\|_{L^2(K_2")}^2.
\end{equation}

For $K_1$  we obtain:
\begin{equation}
	\int_{K_1} \left|u_{per,h}\left(z\right)\right|^2\left|\partial_z \psi_h\left(z\right)\right|^2dz \lesssim  \|u_{per,h}\|_{L^2(K_1")}^2.
\end{equation}

Finally, we can apply these results to our two cases:

\begin{itemize}
	\item \textbf{Interior case}:
	We use the facts $B_{\tilde{M}}$ and $B(0,R_0')\setminus B(0,R')$ are subset of $B(0,R_0)$ and $B(0,R')\subset B(0,R)$:
	\begin{multline}
		e^{-C' \frac{log((h\epsilon(h))^{-1})}{h\epsilon(h)}}\left(\int_{B_{\tilde{M}}\setminus B(0,R)\bigcup_j D'_{j,h}}  |u_{per,h}(z)|^2 dz +o_c(1)\|u_{per,h}\|_{L^2(B(0,R_0))}^2 \right) \\   
		 \lesssim \int_{B(0,R)} \left|u_{per,h}(z)\right|^2dz+e^{-\tilde{C}' \frac{log((h\epsilon(h))^{-1})}{h\epsilon(h)}}\int_{B(0,R_0)} \left|u_{per,h}(z)\right|^2dz.
	\end{multline}
	\item \textbf{Boundary case}: We use the facts  $B(0,R_1)\subset B(0,R_2)$ and $A(0,R_1,R_2)\subset A(0,R,R_0)$:
	\begin{multline}
		e^{-C'  \frac{log((h\epsilon(h))^{-1})}{h\epsilon(h)}}\left(\int_{B(0,R_1)\setminus \bigcup_j D'_{j,h}} |u_{per,h}(z)|^2 dz+ o_c(1) \|u_{per,h}\|_{L^2(B(0,R_2))}^2 \right) \\
		\lesssim \int_{A(0,R,R_0)} \left|u_{per,h}(z)\right|^2\left|\partial_z \psi_h\left(z\right)\right|^2dz.
	\end{multline}
\end{itemize}

\subsection{Carleman's inequality  on contracting balls}
\label{subsect: Carleman_classique}
In this section, we want to get rid of the holes that we create for the proof above. The holes are balls  $D'_{j,h}$ with radius $3 \times 32^2 h\epsilon(h)$, separated by a distance larger than $\tilde{C}_1 h\epsilon(h)$ with $\tilde{C}_1 \gg 1$.
On these $D'_{j,h}$, we know only that $u_{per,h}$ is a solution of $(-h^2\Delta + V_{per} - E)u_{per,h} = 0$ with $V_{per} \in L^\infty(B(0,R_0),\R)$.

Thanks to the contracting character in $h\epsilon(h)$, we can obtain this Carleman inequality:

\begin{thm}
\label{thm: Carlm contractant}
Let $u_h$ be a solution of $(-h^2\Delta + V - E)u_h = 0$ with $E \in I$ compact subset of $\mathbb{R}$, $V \in L^\infty(2D_h,\R)$ and $2D_h$ is a ball with radius $\sim h\epsilon(h)$. Then there is $C_0 > 0$ independent of $h$ and $u_h$ such that for $h > 0$ small enough:
			\begin{equation}
				\int_{D_{h}} |u_h(x)|^2dx \leq C_0 \int_{A_{h}} |u_h(x)|^2dx 
			\end{equation}

where $A_{h} = 2D_{h} \setminus \D_{h}$.
\end{thm}
It is interesting to note that unlike the classic Carleman inequalities, this one has no exponential weight.

We can apply the above theorem to $u_{per,h}$ on $D'_{j,h}$ to control $L^2$-norm on $D'_{j,h}$ by the $L^2$-norm on the set $A_{j,h}=2D'_{j,h}\setminus D'_{j,h}$. In order to treat the balls $\D'_{j,h}$ separately, we choose $\tilde{C}_1$ large enough (Section~\ref{subsect: retourarriere}) such that the rings $A_{j,h}$ are disjointed, i.e. if $i \neq j$, $A_{j,h} \cap A_{i,h} = \emptyset$. Thus, we have:
\begin{itemize}
    \item \textbf{Interior case}:
		\begin{align}
			\int_{B_{\tilde{M}}\setminus B(0,R)} | u_{per,h}(x)|^2 dx &\leq  \int_{B_{\tilde{M}}\setminus B(0,R)\bigcup D'_{j,h}} | u_{per,h}(x)|^2 dx + \sum_{j}\int_{D'_{j,h}} | u_{per,h}(x)|^2 dx\nn\\
			&\leq \int_{B_{\tilde{M}}\setminus B(0,R)\bigcup D'_{j,h}} | u_{per,h}(x)|^2 dx + \sum_{j}C_0 \int_{A_{j,h}}  |u_{per,h}(x)|^2 dx\nn\\
			&\lesssim  \int_{B_{ \tilde{M}}\setminus B(0,R)\bigcup D'_{j,h}} | u_{per,h}(x)|^2 dx + C_0 \int_{\Gamma}  |u_{per,h}(x)|^2 dx
		\end{align}
    where $\Gamma$ is a set of disjoint pieces of rings outside $B_{\tilde{M}} \setminus B(0, R)$. By construction and periodicity of $u_{per,h}$, $\Gamma$ is included in $B(0, R)$.
    \item \textbf{Boundary case}:
		\begin{align}
			\int_{B(0,R_1)} | u_{per,h}(x)|^2 dx &\leq  \int_{B(0,R_1)\setminus \bigcup D'_{j,h}} | u_{per,h}(x)|^2 dx + \sum_{j}\int_{D'_{j,h}} | u_{per,h}(x)|^2 dx\nn\\
			&\leq \int_{B(0,R_1)\setminus \bigcup D'_{j,h}} | u_{per,h}(x)|^2 dx + \sum_{j}C_0 \int_{A_{j,h}}  |u_{per,h}(x)|^2 dx\nn\\
			&\lesssim  \int_{B(0,R_1)\setminus \bigcup D'_{j,h}} | u_{per,h}(x)|^2 dx +  \int_{\Gamma}  |u_{per,h}(x)|^2 dx
		\end{align}

   where $\Gamma$ is a set of disjoint pieces of rings outside $B(0, R_1)$. By construction, $\Gamma_S$ is included in $A(0, R_1, R_2)$.
\end{itemize}

\begin{proof}
See Appendix~\ref{subsect: AP_Carl_small}.
\end{proof}

\section{Conclusion}
\label{subsect: Conclusion}

In this section, we finish the proof by filling the holes and returning to Riemann surfaces $M$:
$u_{h}\in L^2(M,\R)$ is solution of $$(-h^2\Delta+V-E)u_{h}=0\quad \mbox{on } M$$ 
where $V\in L^\infty (M,\R)$ and $E\in I$ compact subset of $\R$. We search to control the $L^2$-norm of $u_h$ on $M$ by the $L^2$-norm of $u_h$ on $U$ an open subset of $M$.

In the section \ref{subsect: Uniformization} we have replace $u_{h}$ by $u_{per,h} \in L^2(B(0,R_0),\R)$ solution of $$(-h^2\Delta+V_{per}-E)u_{per,h}=0\quad \mbox{on } B(0,R_0)$$
with $V_{per} \in L^\infty (B(0,R_0),\R)$. Moreover $u_{per,h}$ is in certain sense periodic and all its values are obtained on a compact set $\tilde{M} \subset B(0,R_0)$ which represent $M$.  
We also concluded that there was two cases to consider:
\begin{itemize}
	\item \textbf{Interior case:} $U$ is replaced by $B(0, R)$  a small ball with radius $R$ centred at $0$ such that $B(0, R) \subset \tilde{M}$. In Section~\ref{subsect: Gag-Nir-Caccio}, we have obtain, after control of deformations cause by the mappings that we use in section \ref{sect: de SchrodAHarmonique}, an inequality for $u_{per,h}$ :
    
		\begin{multline}
			e^{-C' \frac{log((h\epsilon(h))^{-1})}{h\epsilon(h)}}\left(\int_{B_{\tilde{M}}\setminus B(0,R)\bigcup_j D'_{j,h}}  |u_{per,h}(z)|^2 dz +o_c(1)\|u_{per,h}\|_{L^2(B(0,R_0))}^2 \right) \\   
			\lesssim \int_{B(0,R)} \left|u_{per,h}(z)\right|^2dz+e^{-\tilde{C}' \frac{log((h\epsilon(h))^{-1})}{h\epsilon(h)}}\int_{B(0,R_0)} \left|u_{per,h}(z)\right|^2dz
		\end{multline}
    for $h$ small enough, where $\tilde{C}' > C' > 0$ and $\epsilon(h)=\frac{c}{\log(R_0/h)}$ with $c>0$ a small parameter. $D'_{j,h}$ are balls with radius $3 \times 32^2 h\epsilon(h)$ separated by a distance $\tilde{C}_1 h\epsilon(h)$ with $\tilde{C}_1 \gg 1$. $B_{\tilde{M}}$ is a ball containing $\tilde{M}$. In Section~\ref{subsect: Carleman_classique}, we have obtained an another control for $u_{per,h}$: 
		\begin{align}
			\int_{B_{\tilde{M}}\setminus B(0,R)} | u_{per,h}(x)|^2 dx &\lesssim \int_{B_{ \tilde{M}}\setminus B(0,R)\bigcup D'_{j,h}} | u_{per,h}(x)|^2 dx +\int_{\Gamma}  |u_{per,h}(x)|^2 dx
		\end{align}
    for $h$ small enough, where $\Gamma$ is a subset of $B(0, R)$.
    We can combine the two inequalities to get: 
    
    \begin{multline}
    	e^{-C' \frac{log((h\epsilon(h))^{-1})}{h\epsilon(h)}}\left(\int_{B_{\tilde{M}}}  |u_{per,h}(z)|^2 dz +o_c(1)\|u_{per,h}\|_{L^2(B(0,R_0))}^2\right) \\   
    	\lesssim 3\int_{B(0,R)} \left|u_{per,h}(z)\right|^2dz+e^{-\tilde{C}' \frac{log((h\epsilon(h))^{-1})}{h\epsilon(h)}}\int_{B(0,R_0)} \left|u_{per,h}(z)\right|^2dz
    \end{multline}
    for $h$ small enough.
    Thanks to the periodicity of $u_{per,h}$, we can reduce $B(0,R_0)$ to $B_{\tilde{M}}$. Moreover, $0<C'<\tilde{C}'$, so $$ -\tilde{C}' \frac{log((h\epsilon(h))^{-1})}{h\epsilon(h)}<-C' \frac{log((h\epsilon(h))^{-1})}{h\epsilon(h)},$$ and if we choose $c$ small enough, then we can absorb the integrals on $B(0,R_0)$ by the one on $B_{\tilde{M}}$. Thus we obtain:
    
    \begin{equation}
    	e^{-C' \frac{log((h\epsilon(h))^{-1})}{h\epsilon(h)}}\int_{B_{\tilde{M}}}  |u_{per,h}(z)|^2 dz  \\   
    	\lesssim 3\int_{B(0,R)} \left|u_{per,h}(z)\right|^2dz
    \end{equation}
for $h$ small enough.
Finally we can use the Cacciopoli lemma \ref{prop: Caccio} to recover the $L^2$-norm of the gradient:
\begin{equation}
	\int_{\tilde{M}} |h\nabla u_{per,h}|^2\lesssim \int_{B_{\tilde{M}}} |u_{per,h}|^2
\end{equation}
because $\tilde{M}\subset B_{\tilde{M}}$. Thus:
\begin{multline}
	e^{-C' \frac{log((h\epsilon(h))^{-1})}{h\epsilon(h)}}\int_{\tilde{M}} \Big( |u_{per,h}(z)|^2+|h\nabla u_{per,h}(z)|^2\Big) dz  \\   
	\lesssim 3\int_{B(0,R)} \Big( |u_{per,h}(z)|^2+|h\nabla u_{per,h}(z)|^2\Big)dz
\end{multline}
 By definition $\epsilon(h) = \frac{c}{\log(R_0/h)}$, so we obtain:
		\begin{multline}
			\int_{\tilde{M}} \Big( |u_{per,h}(z)|^2+|h\nabla u_{per,h}(z)|^2\Big)dz  \\
			\leq e^{C_0\log(1/h)^{2}/h}\int_{B(0,R)} \Big( |u_{per,h}(z)|^2+|h\nabla u_{per,h}(z)|^2\Big)dz
		\end{multline}
    with $C_0 > 0$. Finally, by definition of $\tilde{M}$, we reduce to $M$ and we obtain for $h$ small enough:
		\begin{equation}
			\int_{M} \Big(|u_{h}|^2+ |h\nabla u_{h}|^2 \Big)d\mu_g
			\leq e^{C_0\log(1/h)^{2}/h}\int_{U} \Big(|u_{h}|^2+ |h\nabla u_{h}|^2 \Big)d\mu_g
		\end{equation}
    with $C_0 > 0$ and $\mu_g$ the volume form on $(M,g)$ (see section \ref{subsect: Uniformization}).
\item \textbf{Boundary case:} $U$ is replaced by $ A(0, R, R_0)$  a ring centred at $0$ with $0 < R < R_0$.

 In Section~\ref{subsect: Gag-Nir-Caccio}, we have obtain, after control of deformations cause by the mappings that we use in section \ref{sect: de SchrodAHarmonique}, an inequality for $u_{per,h}$ :

\begin{multline}
	e^{-C'  \frac{log((h\epsilon(h))^{-1})}{h\epsilon(h)}}\left(\int_{B(0,R_1)\setminus \bigcup_j D'_{j,h}} |u_{per,h}(z)|^2 dz+ o_c(1) \|u_{per,h}\|_{L^2(B(0,R_2))}^2 \right) \\
	\lesssim \int_{A(0,R,R_0)} \left|u_{per,h}(z)\right|^2\left|\partial_z \psi_h\left(z\right)\right|^2dz
\end{multline}
for $h$ small enough, where $ C' > 0$, $R<R_1<R_2<R_0$ and $\epsilon(h)=\frac{c}{\log(R_0/h)}$ with $c>0$ a small parameter. $D'_{j,h}$ are balls with radius $3 \times 32^2 h\epsilon(h)$ separated by a distance $\tilde{C}_1 h\epsilon(h)$ with $\tilde{C}_1 \gg 1$.  In Section~\ref{subsect: Carleman_classique}, we have obtained an another control for $u_{per,h}$: 
\begin{align}
	\int_{B(0,R_1)} | u_{per,h}(x)|^2 dx \lesssim  \int_{B(0,R_1)\setminus \bigcup D'_{j,h}} | u_{per,h}(x)|^2 dx +  \int_{\Gamma}  |u_{per,h}(x)|^2 dx
\end{align}
for $h$ small enough, where $\Gamma$ is a subset of $A(0, R_1,R_2)$.
We can combine the two inequalities to get: 

\begin{multline}
	e^{-C'  \frac{log((h\epsilon(h))^{-1})}{h\epsilon(h)}}\left(\int_{B(0,R_0)} |u_{per,h}(z)|^2 dz+ o_c(1) \|u_{per,h}\|_{L^2(B(0,R_2))}^2 \right) \\
\lesssim 3\int_{A(0,R,R_0)} \left|u_{per,h}(z)\right|^2\left|\partial_z \psi_h\left(z\right)\right|^2dz
\end{multline}
for $h$ small enough.
If we choose $c$ small enough, then we can absorb the integrals on $B(0,R_2)$ by the one on $B(0,R_0)$. Thus we obtain:

\begin{equation}
	e^{-C' \frac{log((h\epsilon(h))^{-1})}{h\epsilon(h)}}\int_{B(0,R_0)}  |u_{per,h}(z)|^2 dz     
	\lesssim 3\int_{B(0,R)} \left|u_{per,h}(z)\right|^2dz
\end{equation}
for $h$ small enough. By definition $\epsilon(h) = \frac{c}{\log(R_0/h)}$ and $\tilde{M}\subset B_{\tilde{M}}$. So we obtain:
\begin{equation}
	\int_{B(0,R_0)} |u_{per,h}(x)|^2dx
	\leq e^{C_0\log(1/h)^{2}/h}\int_{A(0,R,R_0)} |u_{per,h}(x)|^2 dx
\end{equation}
with $C_0 > 0$.
We can then recover the $L^2$-norm on the gradient by using the Cacciopoli lemma \ref{prop: Caccio}:
\begin{equation}
	\int_{B(0,R)}|h\nabla u_{per,h}|^2\lesssim \int_{B(0,R_0)}|u_{per,h}|^2. 
\end{equation}
Thus:
\begin{multline}
		\int_{B(0,R_0)} \Big(|u_{per,h}(x)|^2+ |h\nabla u_{per,h}(x)|^2 \Big)dx\\
	\leq e^{C_0\log(1/h)^{2}/h}\int_{A(0,R,R_0)} \Big(|u_{per,h}(x)|^2+ |h\nabla u_{per,h}(x)|^2 \Big)dx
\end{multline}
 Finally, for this case, we reduce to $\S^2$ by the stereographic projection and we obtain for $h$ small enough:
\begin{equation}
	\int_{\S^2\setminus B(R/2)} \Big(|u_{h}|^2+ |h\nabla u_{h}|^2 \Big)d\sigma
	\leq e^{C_0\log(1/h)^{2}/h}\int_{B(R)\setminus B(R/2)} \Big(|u_{h}|^2+ |h\nabla u_{h}|^2 \Big)d\sigma
\end{equation}
with $C_0 > 0$, $B(R/2)\subset B(R)$ ball centred on the south pole of $\S^2$ inside $U$ and $\sigma$ the area measure on $\S^2$. Thus
\begin{equation}
	\int_{\S^2} \Big(|u_{h}|^2+ |h\nabla u_{h}|^2 \Big)d\sigma
	\leq e^{C_0\log(1/h)^{2}/h}\int_{U} \Big(|u_{h}|^2+ |h\nabla u_{h}|^2 \Big)d\sigma
\end{equation}
and we can reduce to the Riemannian surface $M$:
\begin{equation}
	\int_{M}\Big(|u_{h}|^2+ |h\nabla u_{h}|^2 \Big)d\mu_g
	\leq e^{C_0\log(1/h)^{2}/h}\int_{U} \Big(|u_{h}|^2+ |h\nabla u_{h}|^2 \Big)d\mu_g
\end{equation}
for $h$ small enough, with $\mu_g$ the volume form on $(M,g)$ (see \ref{subsect: Uniformization}).

\end{itemize}

\appendix
\section{Small Poincaré constant}
\label{subsect: small_PC}
We remind that $F_{0,h}$ is the zero set of $u_{per,h}$ inside $B(0,R_0)$ and $F_{1,h}$ is a maximal set of balls inside $B(0,R_0)$ with radius $h\epsilon(h)$ and separated from themselves, $F_{0,h}$ and $\partial B(0,R_0)$ by a distance larger than $Ch\epsilon(h)$ with $C>1$. 
	\begin{lem}
		For $\epsilon(h)$ small enough, the Poincaré constant of $\Omega_h:=B(0,R_0)\setminus F_{0,h}\cup F_{1,h}$ is smaller than $C'h^2\epsilon(h)^2$ where $C'$ is a constant only dependent of $C$.
	\end{lem}
	
	\begin{proof}
		Let $f \in \mathcal{C}_0^\infty(\Omega_h)$. We extend $f$ by $0$ on $B(0,R_0)$. Let $z \in  F_{0,h}\cup F_{1,h} $. For $\epsilon(h)<r_0$ with $r_0$
define in the lemma \ref{lem: positivity_radius}, we know that for all $0<r<h\epsilon(h)$, there exist $z_1 \in C(z,r)\cap ( F_{0,h}\cup F_{1,h})$. So
		
		\begin{equation}
			\max_{C(z,r)} |f|\leq \int_{C(z,r)}|\nabla f|.
		\end{equation}

		\begin{align}
			&\int_{B(z,h\epsilon(h))} |f|^2= \int_0^1 \left(\int_{C(z,r)}|f|^2\right) dr\leq \int_0^1 |C(z,r)|\max_{C(z,r)} |f|^2\leq \int_0^1 |C(z,r)|\left(\int_{C(z,r)}|\nabla f|\right)^2\nn\\
			&\leq \int_0^1 |C(z,r)|^2\left(\int_{C(z,r)}|\nabla f|^2\right)\leq C_1h^2\epsilon(h)^2 \int_{B(z,h\epsilon(h))}|\nabla f|^2
		\end{align}
		
		We therefore can find $r\in (h\epsilon(h)/2,h\epsilon(h))$ such that:
		
		\begin{equation}
			\int_{C(z,r)}|f|^2\leq \frac{C_2}{h\epsilon(h)} \int_{B(z,h\epsilon(h))} |f|^2\leq C_3h\epsilon(h) \int_{B(z,h\epsilon(h))}|\nabla f|^2.
		\end{equation}
		
		Let be $\Gamma_\theta$ the segment starting at the point $x_\theta=z+re^{i\theta}$ and ending at the point $z+3Ch\epsilon(h)e^{i\theta}$.
		We have:
		
		\begin{align}
			&\max_{x\in \Gamma_\theta} |f|^2\leq\left(|f(x_\theta)|+\int_{\Gamma_\theta} |\nabla f| \right)^2\leq 2 \left(|f(x_\theta)|^2+\left(\int_{\Gamma_\theta} |\nabla f|\right)^2 \right)\nn\\
			&\leq  2 \left(|f(x_\theta)|^2+|\Gamma_\theta|\int_{\Gamma_\theta} |\nabla f|^2 \right)\leq 2 \left(|f(x_\theta)|^2+C_4 h\epsilon(h)\int_{\Gamma_\theta} |\nabla f|^2 \right)
		\end{align}
		
		So
		
		\begin{align}
			&\int_{B(z,3Ch\epsilon(h))} |f|^2= \int_{B(z,h\epsilon(h))} |f|^2+\int_{B(z,3Ch\epsilon(h))\setminus B(z,h\epsilon(h))} |f|^2 \nn\\
			&\leq C_1h^2\epsilon(h)^2 \int_{B(z,h\epsilon(h))}|\nabla f|^2 +|\Gamma_\theta|\int_0^{2\pi}\max_{x\in \Gamma_\theta} |f|^2\nn\\
			&\leq C_1h^2\epsilon(h)^2 \int_{B(z,h\epsilon(h))}|\nabla f|^2 +2|\Gamma_\theta| \int_0^{2\pi}\left(|f(x_\theta)|^2+C_4 h\epsilon(h)\int_{\Gamma_\theta} |\nabla f|^2 \right)\nn\\
			&\leq C_5 h^2\epsilon(h)^2 \int_{B(z,3Ch\epsilon(h))} |\nabla f|^2
		\end{align}
		
		Because $B(0,R_0)$ is compact and $F_{0,h}\cup F_{1,h}$ is $3Ch\epsilon(h)$ dense in $B(0,R_0)$, we can choose a finite collection of point in $F_{0,h}\cup F_{1,h}$ such that $B(z,3Ch\epsilon(h))$ cover $B(0,R_0)$ and each point is covered a bounded number of times. 
	\end{proof}
\section{The Toy problem}
\label{subsect: Toy problem}
\begin{thm}
	Let be $((D_{j,h})_j)_h$ a family ball collection on the plan, $32C_1 h\epsilon(h)$ separated (with $C_1\gg 100$), with radius $32h\epsilon(h)$. Let be $\tilde{R}_0\gg1$.  
	\begin{itemize}
		\item \textbf{Interior case:}\\
		Let be $0<\tilde{R}<\tilde{R}_0$,
		there is $\tilde{C}'>C'>0$, such that for $h$ small enough, for any family $(\tilde{f}_h)_h$ of harmonics functions on $B(0,\tilde{R}_0)\setminus \bigcup_j D_{j,h}$  such that $\tilde{f}_h$ does not change sign in each $B(0,\tilde{R}_0)\cap 5D_{j,h}\setminus D_{j,h}$:
		
		\begin{align}
			&\int_{B(0,\tilde{R})\setminus \bigcup_j(3D_{j,h})} \tilde{f}_h^2+e^{-\tilde{C}'\frac{\log \left((h\epsilon(h))^{-1}\right)}{h\epsilon(h)}}\int_{B(0,\tilde{R}_0)\setminus B(0,\frac{1}{2}\tilde{R}_0)\bigcup_j (3D_{j,h})} \tilde{f}_h^2 \nn\\
			&\geq e^{-C'\frac{log\left((h\epsilon(h))^{-1}\right)}{h\epsilon(h)}}\int_{B(0,\frac{1}{2}\tilde{R}_0)\setminus B(0,\tilde{R})\bigcup_j(3D_{j,h})} \tilde{f}_h^2 .
		\end{align}
		\hfill\\
		\item \textbf{Boundary case:}\\
		Let be $0<\tilde{R}_1<\tilde{R}_2<\tilde{R}_0$, there is $C'>0$, such that for $h$ small enough, for any family $(\tilde{f}_h)_h$ of harmonics functions on $B(0,\tilde{R}_0)\setminus \bigcup_j D_{j,h}$ such that $\tilde{f}_h$ does not change sign in each $B(0,\tilde{R}_0)\cap 5D_{j,h}\setminus D_{j,h}$:
		\begin{equation}
			\int_{B(0,\tilde{R}_2)\setminus B(0,\tilde{R}_1)\bigcup_j(3D_{j,h})} \tilde{f}_h^2\geq e^{-C'\frac{\log \left((h\epsilon(h))^{-1}\right)}{h\epsilon(h)}}\int_{B(0,\tilde{R}_1)\setminus \bigcup_j(3D_{j,h})} \tilde{f}_h^2. 
		\end{equation}
	\end{itemize}
\end{thm}
	
	\begin{proof}
		For the proof, we suppose that balls $D_{j,h}$ have radius of $h\epsilon(h)$ (we just do a dilatation by constant factor). So they are separated by a distance larger than $100h\epsilon(h)$. 
		
		Let us start by showing this Carleman inequality:
		
		\begin{lem}(Carleman inequality)\\
			\label{lem: Carleman classique}
			\begin{itemize}
				\item \textbf{Interior case:}\\
				There is $C_0>0$, such that for $h$ small enough, for all $u \in \mathcal{C}_0^\infty(B(0,\tilde{R}_0)\setminus \frac{1}{2}B(0,\tilde{R}))$, we have:
				\begin{equation}
					C_0 k(h)^{-1} \int_{B(0,\tilde{R}_0)} u(x)^2 e^{2k(h)\omega(|x|)}dx\leq \int_{B(0,\tilde{R}_0)} |k(h)^{-2}\Delta u(x)|^2 e^{2k(h)\omega(|x|)}dx
				\end{equation}
				where $\omega(|x|)=|x|^{-1}$, and $k(h)$ a positive function of $h$ which is supposed to increase.
				\hfill\\
				\item \textbf{Boundary case:}\\
				There is $C_0>0$, such that for $h$ small enough, for all $u \in \mathcal{C}_0^\infty(B(0,\tilde{R}_0))$, we have:
				\begin{equation}
					C_0 k(h)^{-2}\int_{B(0,\tilde{R}_0)} u(x)^2 e^{2k(h)\omega(x)}dx\leq  \int_{B(0,\tilde{R}_0)} |k(h)^{-2}\Delta u(x)|^2 e^{2k(h)\omega(x)}dx
				\end{equation}
				where $\omega(x)=x_1$, and $k(h)$ a positive function of $h$ which is supposed to increase.
			\end{itemize}
		\end{lem}
		
		\begin{proof}
			\hfill\\

			\begin{itemize}
				\item \textbf{Interior case:}
				
				To begin with, we rewrite $k(h)^{-2}\Delta$ in polar form:
				
				\begin{equation}
					k(h)^{-2}\Delta= k(h)^{-2} \left[\partial_r^2 +\frac{1}{r}\partial_r +\frac{1}{r^2}\Delta_{S^{1}} \right].
				\end{equation}
				
				To obtain the Carleman inequality we will do some symbolic algebra, for that we quantify in $k(h)^{-1}$.

				Because $\omega$ is radial, we only need to look on radials symbols.
				The radial principal symbol of $k(h)^{-2}\Delta$ is:
				
				\begin{equation}
					p(r,\sigma)= \sigma^2 + \frac{1}{r^2}(\Lambda^*)^2
				\end{equation}
				
				Where $\Lambda^*$ is the principal symbol of $\Delta_{S^{1}}$.
				So by composition with $e^{\pm k(h)\omega(r)}$, the radial principal symbol of $e^{k(h)\omega(r)}(k(h)^{-2}\Delta)e^{-k(h)\omega(r)}$ is:
				
				\begin{equation}
					p_{\omega}(r,\sigma)=p(r,\sigma +i\partial_r \omega(r))= (\sigma+i\partial_r\beta)^2+\frac{(\Lambda^*)^2}{r^2}.
				\end{equation}
			
				Because of $\partial_r\omega\neq 0$ on $B(0,\tilde{R}_0)\setminus \frac{1}{2}B(0,\tilde{R})$, the principal symbol cancels when: 
				
					$$\sigma=0; \quad (\partial_r\omega(r))^2=\frac{(\Lambda^*)^2}{r^2}.$$
			
				We want the Poisson bracket $\{Re\, p_{\omega}, Im\, p_{\omega}\}$ to be positive when $p_{\omega}(r,\sigma)=0$.
				
				\begin{align}
					\{Re\, p_{\omega}, Im\, p_{\omega}\}=& 4\left[\sigma^2\partial_r^2\phi_h +\left(\partial_r^2 \phi_h \partial_r \phi_h +\frac{(\Lambda^*)^2}{r^3} \right)\partial_r \phi_h \right],\nn\\
				\end{align}
				
				so when $p_{\omega}(r,\sigma)=0$, the Poisson bracket reduces to:
				
				\begin{equation}
					\{Re\, p_{\omega}, Im\, p_{\omega}\}= 4\left(\partial_r^2 \omega(r) \partial_r \omega(r) +\frac{(\partial_r \omega(r))^2}{r} \right)\partial_r \omega(r). 
				\end{equation}
				
				If we take $\omega(r)= r^{-1}$ with $r\in [\tilde{R}/2,\tilde{R}_0]$, we have:
				
				\begin{equation}
					\{Re\, p_{\omega}, Im\, p_{\omega}\}=4 r^{-7}>4 R_0^{-7},
				\end{equation}
				
				so the Poisson bracket is positive.
				Moreover $|p_{\omega}|^2 \simeq \langle \sigma \rangle^4=(1+|\sigma|^2)^2$ when $\sigma$ is large. Thus on $B(0,\tilde{R}_0)\setminus \frac{1}{2}B(0,\tilde{R})$, we can find $\tilde{C}>0$ and $d>0$ such that, for $h$ small enough :
				
				\begin{equation}
					d|p_{\omega}|^2+ \{Re\, p_{\omega}, Im\, p_{\omega}\} \geq \tilde{C}\langle \sigma \rangle^4.
				\end{equation}
				
				By Garding inequality (see for example \cite{alinhac2007pseudo}), we obtain for $h$ small enough:
				
				\begin{align}
					& k(h)\Biggl\langle \left[\left(e^{k(h)\omega(r)}(k(h)^{-2}\Delta)e^{-k(h)\omega(r)}\right)^*,e^{k(h)\omega(r)}(k(h)^{-2}\Delta)e^{-k(h)\omega(r)}   \right]v,v \Biggr\rangle \nn\\
					&+d\bigg\|e^{k(h)\omega(r)}(k(h)^{-2}\Delta)e^{-k(h)\omega(r)}v\bigg\|_{L^2}^2\geq \frac{\tilde{C}}{2}\|v \big\|_{H^2}^2
				\end{align}
				
				for $v \in \mathcal{C}_0^{\infty}(B(0,R_0)\setminus \frac{1}{2}B(R'))$. 
				Furthermore:
				
				\begin{align}
					&k(h) \Biggl\langle \left[\left(e^{k(h)\omega(r)}(k(h)^{-2}\Delta)e^{-k(h)\omega(r)}\right)^*,e^{k(h)\omega(r)}(k(h)^{-2}\Delta)e^{-k(h)\omega(r)}   \right]v,v \Biggr\rangle\nn\\
					&= k(h)\Bigg\|e^{k(h)\omega(r)}(k(h)^{-2}\Delta)e^{-k(h)\omega(r)}v\Bigg\|_{L^2}^2- k(h)\Bigg\|\left(e^{k(h)\omega(r)}(k(h)^{-2}\Delta)e^{-k(h)\omega(r)}\right)^*v\Bigg\|_{L^2}^2\nn\\
					&\leq  k(h)\Bigg\|e^{k(h)\omega(r)}(k(h)^{-2}\Delta)e^{-k(h)\omega(r)}v\Bigg\|_{L^2}^2.
				\end{align}
				
				So by replacing the $H^2$ norm by an $L^2$ norm, we obtain:
				
				\begin{equation}
					\Big\|e^{k(h)\omega(r)}(k(h)^{-2}\Delta)e^{-k(h)\omega(r)}v \Big\|_{L^2}^2\geq \frac{ C_0}{k(h)}\big\|v \big\|_{L^2}^2
				\end{equation}

				for $v \in \mathcal{C}_0^{\infty}(B(0,\tilde{R}_0)\setminus \frac{1}{2} B(0,\tilde{R}))$.
				Thus, with $v=u\,e^{k(h)\omega(|x|)}$ we proved the first case.
				\hfill\\
				\item\textbf{Boundary case:}\\
				For this case we follow the proof of \cite{logunov2020landisconjectureexponentialdecay}, section $6.1$:
				
				Let be $u \in \mathcal{C}_0^\infty (B(0,\tilde{R}_0))$ and let be $v=u\,e^{k(h)x_1}$, then 
				
				\begin{equation}
					e^{k(h)x_1}\Delta u= \Delta v -k(h)\partial_{x_1}v +\frac{k(h)^2}{4}v.
				\end{equation}
				
				\begin{align}
					\int_{B(0,\tilde{R}_0)} |\Delta u|^2e^{2k(h)x_1}&=\int_{B(0,\tilde{R}_0)} |\Delta v +\frac{k(h)^2}{4}v|^2+\int_{B(0,\tilde{R}_0)} |k(h)\partial_{x_1}v|^2\nn\\
					&-2\int_{B(0,\tilde{R}_0)} (\Delta v +\frac{k(h)^2}{4}v)k(h)\partial_{x_1}v
				\end{align}
				
				Note that $2v\partial_{x_1}v= \partial_{x_1}(v^2)$ and by integrating by parts:
				\begin{equation}
					-\int_{B(0,\tilde{R}_0)} \partial_{x_1}v \Delta v= -\int_{B(0,\tilde{R}_0)} \Delta v \partial_{x_1}v=0.
				\end{equation}
				
				Hence
				\begin{equation}
					\int_{B(0,\tilde{R}_0)} |\Delta u|^2e^{2k(h)x_1}=\int_{B(0,\tilde{R}_0)} |\Delta v +\frac{k(h)^2}{4}v|^2+\int_{B(0,\tilde{R}_0)} |k(h)\partial_{x_1}v|^2\geq \int_{B(0,\tilde{R}_0)} |k(h)\partial_{x_1}v|^2,
				\end{equation}
				
				and by Poincaré’s inequality
				\begin{equation}
					\int_{B(0,\tilde{R}_0)} |\Delta u|^2e^{2k(h)x_1}\geq C_0 k(h)^2 \int_{B(0,\tilde{R}_0)} u^2 e^{2k(h)x_1}.
				\end{equation}
				
				Thus
				
				\begin{equation}
					\int_{B(0,\tilde{R}_0)} |k(h)^{-2}\Delta u|^2e^{2k(h)x_1}\geq C_0 k(h)^{-2} \int_{B(0,\tilde{R}_0)} u^2 e^{2k(h)x_1}
				\end{equation}
				
				(We could have done a symbolic proof, we would have gain $k(h)^{-1}$ in the right member, but this will not be useful).
			\end{itemize}
		\end{proof}

Thanks to this lemma, the proof of the theorem proceeds as followed: 
		\begin{itemize}
			\item \textbf{Interior case:}\\
			Let be $\eta_h \in \mathcal{C}_0^\infty(B(0,\tilde{R}_0)\setminus B(0,\frac{1}{2}\tilde{R}))$ such that:
			
			\begin{itemize}
				\item $\eta_h$ is non-negative,
				\item $\eta_h=0$ on $2D_{j,h}$, $B(0,\frac{2}{3}\tilde{R})$, and $\{x:|x|\geq \tilde{R}_0-11h\epsilon(h)\}$,
				\item $\eta_h=1$ on $B(0,\frac{3}{4}\tilde{R}_0)\setminus B(0,\frac{3}{4}\tilde{R})\bigcup_j(3D_{j,h})$,
				\item the function $\eta_h$ is bounded, its first derivatives are bounded by a factor $\lesssim{(h\epsilon(h))^{-1}}$ and its second derivatives  by a factor $\lesssim{(h\epsilon(h))^{-2}}$.
			\end{itemize}
			We apply the lemma \ref{lem: Carleman classique} to $u_h=\tilde{f}_h\,\eta_h$, we get:
			\begin{align}
				\label{eq:CarlLogsanscompact}
				&\int_{B(0,\tilde{R}_0)\setminus B(0,\frac{2}{3}\tilde{R})\bigcup_j(2D_{j,h})} |\eta_h k(h)^{-2}\Delta \tilde{f}_h(x)|^2 e^{2k(h)\omega(|x|)}dx + "\mbox{res}"\nn\\
				&\geq \frac{C_0}{k(h)}\int_{B(0,\frac{3}{4}\tilde{R}_0)\setminus B(0,\frac{3}{4}\tilde{R})\bigcup_j(3D_{j,h})} \tilde{f}_h(x)^2 e^{2k(h)\omega(|x|)}dx.
			\end{align}
Since $\tilde{f}_h$ is harmonic, all that remains are the residues which are integrals of the form:
			\begin{align}
				I&=\sum_{5D_{j,h}\subset B(0,\frac{3}{4}\tilde{R}_0)\setminus B(0,\frac{3}{4}\tilde{R})}\int_{3D_{j,h}\setminus2D_{j,h}} "\mbox{Cut-off res}"\\
				II&=\int_{B(0,\tilde{R}_0-11h\epsilon(h))\setminus B(0,\frac{3}{4}\tilde{R}_0-10h\epsilon(h))\bigcup_j(2D_{j,h})}"\mbox{Cut-off res}"\\
				III&= \int_{B(0,(\frac{3}{4}+10h\epsilon(h))\tilde{R})\setminus B(0,\frac{2}{3}\tilde{R})\bigcup_j(2D_{j,h})}"\mbox{Cut-off res}"
			\end{align}
			
			We suppose that $k(h)^{-1}$ decrease faster than $h\epsilon(h)$, so each “Cut-off res” is $\lesssim(\tilde{f}_h^2 +|k(h)^{-1} \nabla \tilde{f}_h|^2)e^{2k(h)\omega}$.
			
			 Note that if $5D_{j,h}\subset B(0,\frac{3}{4}\tilde{R}_0)\setminus B(0,\frac{3}{4}\tilde{R})$,  there is $\tilde{C}>0$ such that:
			
			\begin{equation}
				\int_{3D_{j,h}\setminus 2D_{j,h}}e^{2k(h)\omega(|x|)}dx\lesssim e^{-\tilde{C}\,k(h)h\epsilon(h)}\int_{4D_{j,h}\setminus 3D_{j,h}}e^{2k(h)\omega(|x|)}dx
			\end{equation}
			
			Because $4D_{j,h}\setminus 3D_{j,h}$ contains a ball $\tilde{B}$ with radius $h\epsilon(h)/4$ where
			$$ \inf_{\tilde{B}}e^{2k(h)\omega(|x|)}\geq e^{\tilde{C}\,k(h)h\epsilon(h)}\sup_{3D_{j,h}\setminus 2D_{j,h}}e^{2k(h)\omega(|x|)}$$
			(Figure~\ref{fig:ToyDemo}). 
			
			\begin{figure}
				\centering
				\includegraphics[width=0.6\textwidth]{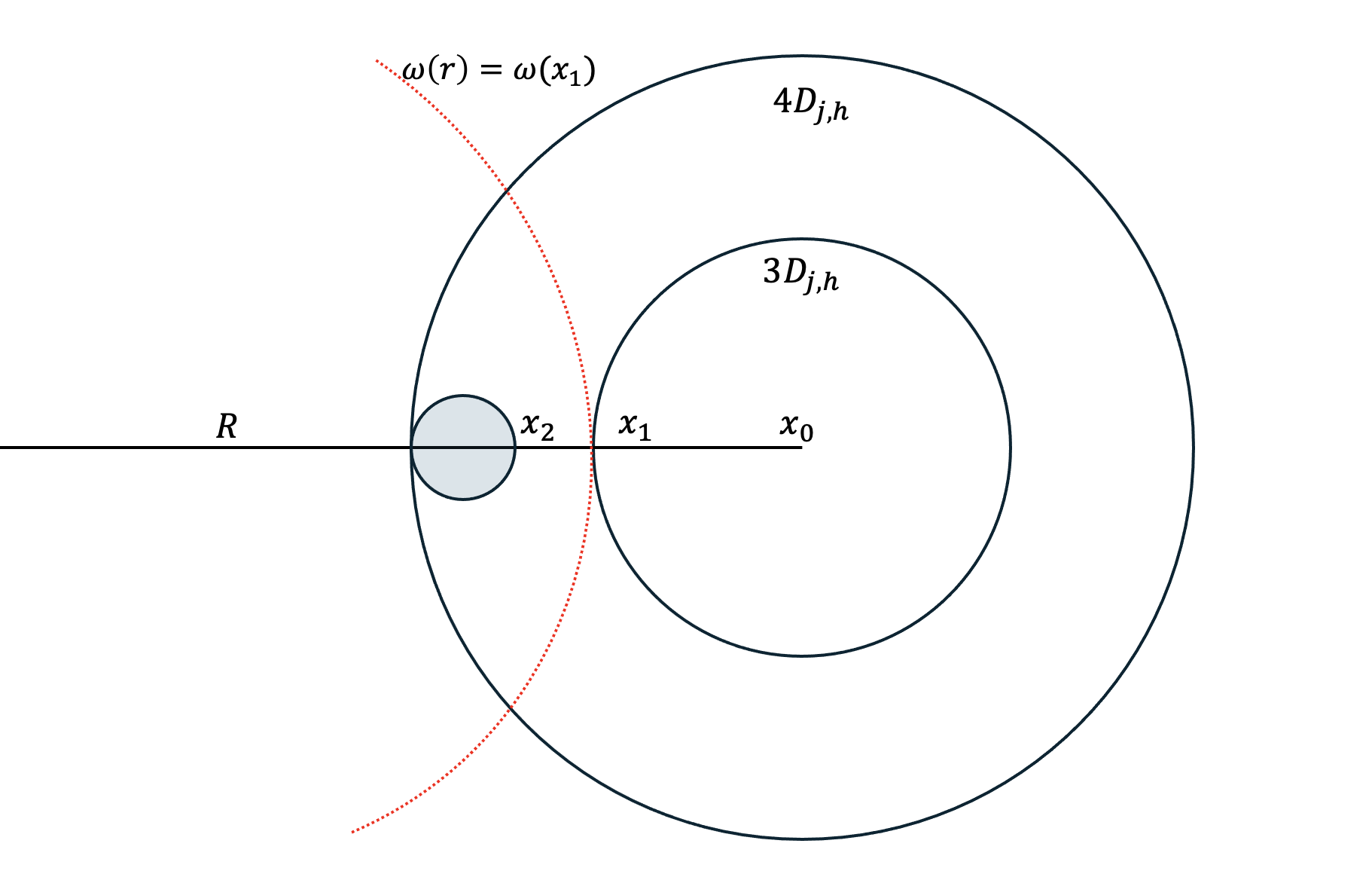}
				\caption{ On $\tilde{B}$ (the blue ball), $e^{2k(h)|x|^{-1}}$ is bigger than $e^{2k(h)|x_2|^{-1}}$, which is bigger than $e^{\tilde{C}\,k(h)h\epsilon(h)}e^{2k(h)|x_1|^{-1}}$ (where $|x_1|^{-1}=\sup_{3D_{j,h}} |x|^{-1}$) }
				\label{fig:ToyDemo}
			\end{figure}

			Assume that $5D_{j,h}\subset B(0,\frac{3}{4}\tilde{R}_0)\setminus B(0,\frac{3}{4}\tilde{R})$. Since $\tilde{f}_h$ does not change sign on $5D_{j,h}\setminus D_{j,h}$, by Harnack's inequality and Cauchy's estimation there exists a constant $a_{j,h}$ such that $|\tilde{f}_h|\simeq a_{j,h}$ , $|k(h)^{-1}\nabla \tilde{f}_h|\lesssim a_{j,h}(h\epsilon(h)k(h))^{-1}\lesssim a_{j,h}$ on $4D_{j,h}\setminus 2D_{j,h}$.
			So:

			\begin{align}
				\int_{3D_{j,h}\setminus 2D_{j,h}}"\mbox{Cut-off res}"&\lesssim  a_{j,h}^2\int_{3D_{j,h}\setminus 2D_{j,h}}e^{2k(h)\omega(|x|)}dx \nn\\
				&\lesssim a_{j,h}^2e^{-\tilde{C}\,k(h)h\epsilon(h)}\int_{4D_{j,h}\setminus 3D_{j,h}}e^{2k(h)\omega(|x|)}dx \nn\\
				&\lesssim e^{-\tilde{C}\,k(h)h\epsilon(h)}\int_{4D_{j,h}\setminus 3D_{j,h}}\tilde{f}_h(x)^2 e^{2k(h)\omega(|x|)}dx.
			\end{align}
			
			Thus, because all ball $D_{j,h}$ are $100h\epsilon(h)$-separated, we get
			
			\begin{equation}
				I\lesssim e^{-\tilde{C}\,k(h)h\epsilon(h)}\int_{B(0,\frac{3}{4}\tilde{R}_0)\setminus B(0, \frac{3}{4}\tilde{R})\bigcup_j(3D_{j,h})} \tilde{f}_h(x)^2e^{2k(h)\omega(|x|)}dx.
			\end{equation}
			
			Note then that 
			
			\begin{equation}
				\frac{C_0}{k(h)}\int_{B(0,\frac{3}{4}\tilde{R}_0)\setminus B(0,\frac{3}{4}\tilde{R})\bigcup_j(3D_{j,h})} \tilde{f}_h(x)^2 e^{2k(h)\omega(|x|)}dx>2I
			\end{equation}
			
			as soon as $\frac{C_0}{k(h)}>>e^{-\tilde{C}\,k(h)h\epsilon(h)}$. We choose then $k(h)=-c h^{-1}\epsilon(h)^{-1}\log(h\epsilon(h))$ with $c>1$ (Note that $k(h)^{1}$ decrease faster than $h\epsilon(h)$). 
			So with the estimate \ref{eq:CarlLogsanscompact}, we have:
			
			\begin{equation}
				2I \leq \frac{C_0}{k(h)}\int_{B(0,\frac{3}{4}\tilde{R}_0)\setminus  B(0,\frac{3}{4}\tilde{R})\bigcup_j(3D_{j,h})} \tilde{f}_h(x)^2 e^{2k(h)\omega(|x|)}dx \leq I + II+III.
			\end{equation}
			
			We deduce a minor bound for $III+II$:

			\begin{align}
				& III+II\geq  \frac{C_0}{2k(h)}\int_{B(0,\frac{3}{4}\tilde{R}_0)\setminus B(0, \frac{3}{4}\tilde{R})\bigcup_j(3D_{j,h})} \tilde{f}_h(x)^2 e^{2k(h)\omega(|x|)}dx\nn\\
				&\geq e^{2\omega(\frac{1}{2}\tilde{R}_0) k(h)}\int_{B(0,\frac{1}{2}\tilde{R}_0)\setminus  B(0,\frac{3}{4}\tilde{R})\bigcup_j(3D_{j,h})} \tilde{f}_h(x)^2dx.
			\end{align}
			
			Note now that if $5D_{j,h}\subset B(0,\tilde{R})$, then $\int_{3D_{j,h}\setminus 1,5D_{j,h}}\tilde{f}_h^2\simeq \int_{4D_{j,h}\setminus 3D_{j,h}}\tilde{f}_h^2$. So for $h$ small enough, by Cauchy estimate:
			
			\begin{align}
				\int_{B(0,\tilde{R})\setminus \bigcup_j(3D_{j,h})}& \tilde{f}_h(x)^2dx  \gtrsim \int_{B(\tilde{R}-10h\epsilon(h))\setminus  \bigcup_j(1,5D_{j,h})}\tilde{f}_h(x)^2dx\nn\\
				&\gtrsim \int_{(\frac{3}{4}+10h\epsilon(h))B(0,\tilde{R})\setminus \frac{2}{3}B(0,\tilde{R}) \bigcup_j(2D_{j,h})} \tilde{f}_h(x)^2+|k(h)^{-1}\nabla \tilde{f}_h(x)|^2dx.
			\end{align}
			
			Thus we obtain a bound for $III$:
			
			\begin{align}
				&\sup_{B(0,\tilde{R})\setminus B(0,\frac{2}{3}\tilde{R})}e^{2k(h)\omega(|x|)}\int_{B(0,\tilde{R})\setminus \bigcup_j(3D_{j,h}))} \tilde{f}_h(x)^2dx\nn\\
				&\gtrsim\int_{(\frac{3}{4}+10h\epsilon(h))B(0,\tilde{R})\setminus B(0,\frac{2}{3}\tilde{R}) \bigcup_j(2D_{j,h})} (\tilde{f}_h(x)^2+|k(h)^{-1}\nabla \tilde{f}_h(x)|^2)e^{2k(h)\omega}dx\nn\\
				&\gtrsim III.
			\end{align}
			
			In the same way, for $5D_{j,h}\subset B(0,\tilde{R}_0)$, $\int_{3D_{j,h}\setminus 1,5D_{j,h}}\tilde{f}_h^2\simeq \int_{4D_{j,h}\setminus 3D_{j,h}}\tilde{f}_h^2$. So for $h$ small enough, by Cauchy estimate:
			
			\begin{align}
				&\int_{B(0,\tilde{R}_0)\setminus B(0,\frac{1}{2}\tilde{R}_0)\bigcup_j 3D_{j,h}} \tilde{f}_h(x)^2dx\gtrsim\int_{B(0,\tilde{R}_0-10h\epsilon(h))\setminus B(0,\frac{1}{2}R_0)\bigcup_j 1,5D_{j,h}} \tilde{f}_h(x)^2dx\nn\\
				&\gtrsim \int_{B(0,\tilde{R}_0-11h\epsilon(h))\setminus B(0,\frac{3}{4}\tilde{R}_0-10h\epsilon(h))\bigcup_j 2D_{j,h}}\tilde{f}_h(x)^2 +|k(h)^{-1}\nabla \tilde{f}_h(x)|^2dx.
			\end{align}
			
			Thus we obtain for $II$:
			
			\begin{equation}
			\sup_{B(0,\tilde{R}_0)\setminus B(0,\frac{3}{4}\tilde{R}_0-10h\epsilon(h))}e^{2k(h)\omega(|x|)}	\int_{B(0,\tilde{R}_0)\setminus B(0,\frac{1}{2}R_0)\bigcup_j 3D_{j,h}} \tilde{f}_h(x)^2dx\gtrsim II.
			\end{equation}
			
			Therefore with $\omega(|x|)=|x|^{-1}$ for $h$ small enough
			
			\begin{align}
			e^{\frac{3k(h)}{\tilde{R}}}&	\int_{B(0,\tilde{R})\setminus \bigcup_j(3D_{j,h})} \tilde{f}_h(x)^2dx+e^{\frac{16k(h)}{5\tilde{R}_0}} \int_{B(0,\tilde{R}_0)\setminus B(0,\frac{1}{2}\tilde{R}_0)\bigcup_j 3D_{j,h}} \tilde{f}_h(x)^2dx\nn\\
			&\gtrsim e^{\frac{4k(h)}{\tilde{R}_0}} \int_{B(0,\frac{1}{2}\tilde{R}_0)\setminus  B(0,\frac{3}{4}\tilde{R})\bigcup_j(3D_{j,h})} \tilde{f}_h(x)^2dx .
			\end{align}
			
			that is to say:
			
			\begin{align}
				&\int_{B(0,\tilde{R})\setminus \bigcup_j(3D_{j,h})} \tilde{f}_h(x)^2dx+e^{-C_2 k(h)} \int_{B(0,\tilde{R}_0)\setminus B(0,\frac{1}{2}\tilde{R}_0)\bigcup_j 3D_{j,h}} \tilde{f}_h(x)^2dx\nn\\
				&\geq e^{-C_1 k(h)} \int_{B(0,\frac{1}{2}\tilde{R}_0)\setminus \frac{3}{4} B(0,\tilde{R})\bigcup_j(3D_{j,h})} \tilde{f}_h(x)^2dx 
			\end{align}
			with $0<C_1<C_2$.
			\hfill\\
			\item \textbf{Boundary case:}\\
			Let be $\eta_h \in \mathcal{C}_0^\infty(B(0,R_0))$ such that:
			\begin{itemize}
				\item $\eta_h$ is non-negative,
				\item $\eta_h=0$ on $2D_{j,h}$ and $\{x:|x|\geq \tilde{R}_2-11h\epsilon(h)\}$,
				\item $\eta_h=1$ on $B(0,\tilde{R}_1+20h\epsilon(h))\setminus \bigcup_j(3D_{j,h})$,
				\item the function $\eta_h$ is bounded its first derivative are bounded by a factor ${(h\epsilon(h))^{-1}}$ and its second derivative  by a factor ${(h\epsilon(h))^{-2}}$.
			\end{itemize}
			
			We apply the second point of the lemma \ref{lem: Carleman classique} to $u_h=\tilde{f}_h \eta_h$ and we get:			
			\begin{align}
				&\int_{B(0,\tilde{R}_2)\setminus \bigcup_j(2D_{j,h})} |\eta_h k(h)^{-2}\Delta \tilde{f}_h(x)|^2 e^{2k(h)\omega(x)}dx + "\mbox{res}"\nn\\
				&\geq \frac{C_0}{k(h)^2}\int_{B(0,\tilde{R}_1+20h\epsilon(h))\setminus \bigcup_j(3D_{j,h})} \tilde{f}_h(x)^2 e^{2k(h)\omega(x)}dx.
			\end{align}
			
Since $\tilde{f}_h$ is harmonic, it remains only $"res"$ in the left term. $"res"$ is integrals of the form :			
			\begin{align}
				I&=\sum_{5D_{j,h}\subset B(0,\tilde{R}_1+20h\epsilon(h))}\int_{3D_{j,h}\setminus 2D_{j,h}} "\mbox{Cut-off res}"\\
				II&=\int_{B(0,\tilde{R}_2-11h\epsilon(h))\setminus B(0,\tilde{R}_1+10h\epsilon(h))\bigcup_j 2D_{j,h}}"\mbox{Cut-off res}".\\
			\end{align}
We suppose that $k(h)^{-1}$ decrease faster than $h\epsilon(h)$ such that 
			 $$"\mbox{Cut-off res}" \lesssim(\tilde{f}_h^2 +|k(h)^{-1} \nabla \tilde{f}_h|^2)e^{2k(h)\omega}$$.
			
Note that there is $\tilde{C}>0$ such that:
			\begin{equation}
				\int_{3D_{j,h}\setminus 2D_{j,h}}e^{2k(h)\omega(x)}\lesssim e^{-\tilde{C}\,k(h)h\epsilon(h)}\int_{4D_{j,h}\setminus 3D_{j,h}}e^{2k(h)\omega(x)}.
			\end{equation}
			
			Because $4D_{j,h}\setminus 3D_{j,h}$ contains a ball $\tilde{B}$ with radius $h\epsilon(h)/4$ where
			$$\sup_{\tilde{B}} e^{2k(h)x_1} \geq e^{\tilde{C}\,k(h)h\epsilon(h)}\sup_{3D_{j,h}\setminus 2D_{j,h}}e^{2k(h)x_1} $$
			(Figure~\ref{fig:ToyDemo2}). 
			
			\begin{figure}
				\centering 
				\includegraphics[width=0.4\textwidth]{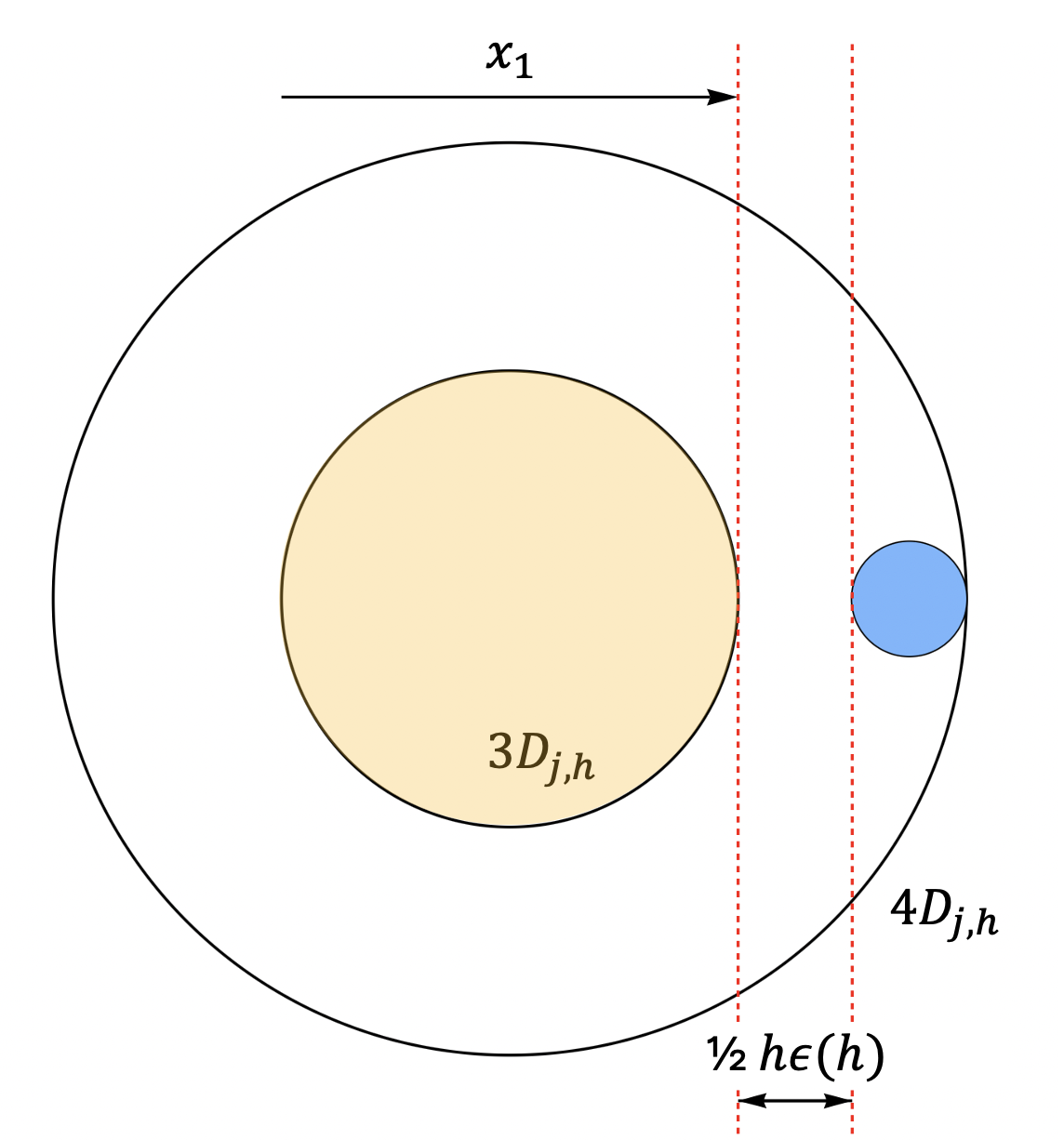}
				\caption{ On $\tilde{B}$ (blue ball), $e^{2k(h)x_1}$ is bigger than $e^{2k(h)x_1}$ on  $3D_{j,h}$ (orange ball) by a factor $e^{\frac{h\epsilon(h)k(h)}{2}}$.}
				\label{fig:ToyDemo2}
			\end{figure}
			
			Assume that $5D_{j,h}\subset B(0,\tilde{R}_1+20h\epsilon(h))$. Since $\tilde{f}_h$ does not change sign on $5D_{j,h}\setminus D_{j,h}$, by Harnack's inequality and Cauchy's estimation there exists a constant $a_{j,h}$ such that, for $h$ small enough, $|\tilde{f}_h|\simeq a_{j,h}$ and $|k(h)\nabla \tilde{f}_h|\lesssim a_{j,h}(h\epsilon(h)k(h))^{-1}\lesssim a_{j,h}$ on $3D_{j,h}\setminus 2D_{j,h}$.
			So:
		
		\begin{align}
			\int_{3D_{j,h}\setminus 2D_{j,h}}"\mbox{Res}"&\lesssim a_{j,h}^2\int_{3D_{j,h}\setminus 2D_{j,h}}e^{2k(h)\omega(x)}dx\\
			&\lesssim a_{j,h}^2e^{-\tilde{C}\,k(h)h\epsilon(h)}\int_{4D_{j,h}\setminus 3D_{j,h}}e^{2k(h)\omega(x)}dx\\
			&\lesssim e^{-\tilde{C}\,k(h)h\epsilon(h)}\int_{4D_{j,h}\setminus 3D_{j,h}}\tilde{f}_h(x)^2 e^{2k(h)\omega(x)}dx.
		\end{align}
		
		Hence, because all $D_{j,h}$ are $100h\epsilon(h)$-separated:
		
		\begin{equation}
			I\lesssim e^{-\tilde{C}\,k(h)h\epsilon(h)}\int_{B(0,\tilde{R}_1+20h\epsilon(h))\setminus \bigcup_j(3D_{j,h})} \tilde{f}_h(x)^2e^{2k(h)\omega(x)}dx.
		\end{equation}
		
		Note then that 
		\begin{equation}
			\frac{C_0}{k(h)^2}\int_{B(0,\tilde{R}_1+20h\epsilon(h))\setminus \bigcup_j(3D_{j,h})} \tilde{f}_h(x)^2 e^{2k(h)\omega(x)}dx>2I
		\end{equation}
		as soon as $\frac{C_0}{k(h)^2} \gg e^{-\tilde{C}\,k(h)h\epsilon(h)}$. We choose then $k(h)=-c h^{-1}\epsilon(h)^{-1}\log(h\epsilon(h))$ with $c>2$ so $k(h)^{-1}$ decrease faster than $h\epsilon(h)$. 
		
		With the precedent inequality we get:		
		\begin{equation}
			2I \leq \frac{C_0}{k(h)}\int_{B(0,\tilde{R}_1+20h\epsilon(h)\setminus\bigcup_j(3D_{j,h})} \tilde{f}_h(x)^2 e^{2k(h)\omega(x)}dx \leq I + II.
		\end{equation}
		
		We deduce a minor bound for $II$: 		
			\begin{align}
				& II\geq  \frac{C_0}{2k(h)}\int_{B(0,\tilde{R}_1+20h\epsilon(h)\setminus\bigcup_j(3D_{j,h})} \tilde{f}_h(x)^2 e^{2k(h)\omega(x)}dx\nn\\
				&\geq e^{-C' k(h)}\int_{B(0,\tilde{R}_1)\setminus\bigcup_j(3D_{j,h})} \tilde{f}_h(x)^2dx.
			\end{align}
			
Now, if $5D_{j,h}\subset B(0,\tilde{R}_2)\setminus B(0,\tilde{R}_1)$, then $\int_{3D_{j,h}\setminus 1,5D_{j,h}}\tilde{f}_h^2\simeq \int_{4D_{j,h}\setminus 3D_{j,h}}\tilde{f}_h^2$. By Cauchy estimate:
			\begin{align}
				\int_{B(0,\tilde{R}_2)\setminus B(0,\tilde{R}_1) \bigcup_j(3D_{j,h})} &\tilde{f}_h(x)^2dx  \gtrsim\int_{B(0,\tilde{R}_2-10h\epsilon(h))\setminus B(0,\tilde{R}_1) \bigcup_j(1,5D_{j,h})}\tilde{f}_h(x)^2dx\nn\\
				&\gtrsim \int_{B(0,\tilde{R}_2-11h\epsilon(h))\setminus B(0,\tilde{R}_1+10h\epsilon(h)) \bigcup_j(2D_{j,h})} \tilde{f}_h(x)^2+|k(h)^{-1}\nabla \tilde{f}_h(x)|^2dx.
			\end{align}
			
			Thus we obtain a bound for $II$:
			
			\begin{align}
				&\sup_{B(0,\tilde{R}_0)}e^{2k(h)\omega(x)}\int_{B(0,\tilde{R}_2)\setminus B(0,\tilde{R}_1) \bigcup_j(3D_{j,h})} \tilde{f}_h(x)^2dx\nn\\
				&\gtrsim\int_{B(0,\tilde{R}_2-11h\epsilon(h))\setminus B(0,\tilde{R}_1+10h\epsilon(h)) \bigcup_j(2D_{j,h})} (\tilde{f}_h(x)^2+|\nabla \tilde{f}_h(x)|^2)e^{2k(h)\omega(x)}dx\nn\\
				&\gtrsim II.
			\end{align}
			
			Hence, with $\omega(x)=x_1$, we get:
			\begin{equation}
			e^{2k(h)R_0}\int_{B(0,\tilde{R}_2)\setminus B(0,\tilde{R}_1) \bigcup_j(3D_{j,h})} \tilde{f}_h^2\geq e^{-C'\,k(h)}\int_{ B(0,\tilde{R}_1)\setminus \bigcup_j(3D_{j,h})} \tilde{f}_h^2.
			\end{equation}
			
			That is to say:
			\begin{equation}
			\int_{B(0,\tilde{R}_2)\setminus B(0,\tilde{R}_1) \bigcup_j(3D_{j,h})} \tilde{f}_h^2\geq e^{-C\,k(h)}\int_{ B(0,\tilde{R}_1)\setminus \bigcup_j(3D_{j,h})} \tilde{f}_h^2
			\end{equation}
			with $C>0$.
		\end{itemize}

	\end{proof}
\section{Control of the Riemann mapping }
\label{subsect: controle de partialz Rh(z)}
\begin{lem}
		\hfill\\
		Let $K$ be a compact set of $B(0,R_0)$ such that $d(K, B(0,R_0)^c)>\delta>0$ with $\delta>0$. For $h$ small enough:
		\begin{itemize}
			\item There is a constant $\tilde{A}>0$ independent of $h$ such that:
			
			\begin{equation}
				\sup_{z\in \psi_h(K)}|\partial_z R_h(z)|\leq \tilde{A}
			\end{equation}
			
			\item There is a constant $\tilde{B}>0$ independent of $h$ such that:
			
			\begin{equation}
				\frac{1}{\tilde{B}}\leq \inf_{z\in K}|\partial_z R_h(z)|
			\end{equation}
			
		\end{itemize}
	\end{lem}
	
	\begin{proof}

		According to proposition \ref{prop: def_de_psi_h} and the Sobolev embedding  $W^{1,3}\subset C^{0,1/3}$, we can show that for $h$ small enough:
		
		\begin{equation}
			||\psi_h- Id||_{L^\infty(B(0,R_0))}\leq R_0^{1/3}||\partial_z\psi_h-1||_{L^3(B(0,R_0))}\longrightarrow 0.
		\end{equation}
		
		So for all $a>0$ small, we have for $h$ small enough, for all $z\in B(0,R_0)$,
		
		\begin{equation}
			\label{eq: eq_sur_psih_pour_Rh}
			|\psi_h(z)-z|\leq a.
		\end{equation} 
		
		\begin{itemize}
			
			\item $R_h$ is a bi-holomorphic mapping from $\psi_h(B(0,R_0))$ to $B(0,R_0)$. According to Cauchy's estimate, for all $\delta>0$, there is a constant $A$ independent of $h$, such that for all compact set $K$ of $B(0,R_0)$ with $d(K, B(0,R_0)^c)>\delta$, we have
			
			\begin{equation}
				\sup_{z\in \psi_h(K)}|\partial_z R_h(z)| \leq A \sup_{z\in \psi_h(B(0,R_0))}|R_h(z)|\leq A R_0
			\end{equation}
			(because $\psi_h$ send $K$ on a compact set $\tilde{K}$ such that  $d(\tilde{K}, \psi_h(B(0,R_0))^c)>\tilde{\delta}=\delta-a>0$ for $h$ small enough).

			\item Since $R_h$ is bi-holomorphic, $|\partial_z R_h|>0$. So by Cauchy's estimate
			
			\begin{align}
				&\sup_{z\in \psi_h(K)}\left|\frac{1}{\partial_z R_h(z)}\right|\nn\\
				&= \sup_{z\in R_h\circ\psi_h(K)}|\partial_z R_h^{-1}(z)|\leq B \sup_{z\in B(0,R_0)} | R_h^{-1}(z)|,
			\end{align}
			
			because by the Mori inequality \ref{eq: InegdeMori}, $g_h=R_h\circ\psi_h(K)$ send $K$ on $\tilde{K}$ such that\\ $d(\tilde{K}, \psi_h(B(0,R_0))^c)>\tilde{\delta}>0$ for $h$ small enough.\\
			But $ R_h^{-1}(z)$ send $B(0,R_0)$ on $\psi_h(B(0,R_0))$. By the equation \ref{eq: eq_sur_psih_pour_Rh}, we have:
			
			\begin{equation}
				\sup_{z\in B(0,R_0)} | R_h^{-1}(z)|=\sup_{z\in B(0,R_0)} | \psi_h(z)|\leq a+R_0.
			\end{equation}
			So:
			
			\begin{equation}
				\sup_{z\in \psi_h(K)}\left|\frac{1}{\partial_z R_h(z)}\right|\leq \tilde{B},
			\end{equation}
			
			equivalently:
			
			\begin{equation}
				\frac{1}{\tilde{B}}\leq \inf_{z\in \psi_h(K)}|R_h(z)|.
			\end{equation}
			
		\end{itemize}
	\end{proof}
\section{Cacciopoli Lemma}
\label{subsect: Cacciopoli}
	\begin{prop}(Cacciopoli)\\
		Let be $\eta \in \mathcal{C}_0^\infty(\Omega)$ with $\Omega$ an open of $\mathbb{R}^2$. We have:
		
		\begin{equation}
			\int_{\Omega}|\eta\nabla u_{per,h}|^2 \leq 6\int_{\Omega}u_{per,h}^2 \left( |\nabla \eta|^2 + \frac{|V_{per}-E|}{h^2}\eta^2 \right)
		\end{equation}
		
	\end{prop}
	
	\begin{proof}
		
		Indeed,
		
		\begin{equation}
			\int_{\Omega}(-h^2\Delta+V_{per}-E)u_{per,h}\; (u_{per,h}\eta^2)=0
		\end{equation}
		
		and 
		
		\begin{align}
			\nabla u_{per,h} \nabla (u_{per,h}\eta^2) &= \nabla u_{per,h} \nabla (u_{per,h}\eta)\,\eta+ \nabla u_{per,h} \nabla \eta\,(u_{per,h}\eta)\nn\\
			&=(\nabla (u_{per,h}\eta) -u_{per,h}\nabla \eta)\nabla(u_{per,h}\eta)+ (\nabla (u_{per,h}\eta) -u_{per,h}\nabla \eta)u_{per,h}\nabla(\eta)\nn\\
			&=|\nabla (u_{per,h}\eta)|^2 -|u_{per,h}\nabla \eta|^2
		\end{align}
		
		By using the inequality $|a+b|^2\geq |a|^2/2-2|b|^2$ with $|\nabla( u_{per,h}\, \eta)|^2 =|\eta\nabla u_{per,h} +u_{per,h}\nabla  \eta|^2$,
		we get:
		
		\begin{equation}
			\int_{\Omega}|\eta\nabla u_{per,h}|^2 \leq 6\int_{\Omega}u_{per,h}^2 \left( |\nabla \eta|^2 + \frac{|V_{per}-E|}{h^2}\eta^2 \right)
		\end{equation}
		
	\end{proof}
\section{Carleman inequality on small ball}
\label{subsect: AP_Carl_small}
\begin{thm}
	Let $u_h$ be a solution of $(-h^2\Delta + V - E)u_h = 0$ with $E \in I$ compact subset of $\mathbb{R}$, $V \in L^\infty(2D_h,\R)$ and $2D_h$ is a ball with radius $\sim h\epsilon(h)$. Then there is $C_0 > 0$ independent of $h$ and $u_h$ such that for $h > 0$ small enough:
	\begin{equation}
		\int_{D_{h}} |u_h(x)|^2dx \leq C_0 \int_{A_{h}} |u_h(x)|^2dx 
	\end{equation}
	
	where $A_{h} = 2D_{h} \setminus \D_{h}$.
	\end{thm}
	
	\begin{proof}
		
		In order to obtain the theorem \ref{thm: Carlm contractant}, we'll use the work of \cite{kloppSemiclassicalResolventEstimates2019a} as a guide and the fact that the balls are shrinking. To simplify writing, we'll assume that $D_h=B(0,h\epsilon(h))$. 
		We start by modifying a little our operator:
		
		\begin{defn}
			
			\begin{equation}
				P_{V,E}:= \left(h\,\delta\right)^{2/3}(-h^2 \Delta+V-E)
			\end{equation}
			
			Where $\delta$ is a new little parameter independent of $h$ which will be fixed later.
		\end{defn}
		
		With $\tilde{h}:=h^{4/3}\delta^{1/3}$, we get:
		
		\begin{equation}
			P_{V,E}=-\tilde{h}^2\Delta +\left(\tilde{h}\delta\right)^{1/2}(V-E).
		\end{equation}
		
		Let $\phi:\mathbb{R}^2 \rightarrow \mathbb{R}$ be a real-valued function that will also be fixed later, we set
		
		\begin{equation}
			P_\phi=e^{\phi/\tilde{h}}P_{V,E} e^{-\phi/\tilde{h}}.
		\end{equation}
		
		Let $v\in \mathcal{C}_0^\infty(B(0,2h\epsilon(h)))$ be a test function, we have this lower bound for $P_\phi$:
		
		\begin{align}
			\label{eq: inég carleman V Linfini}
			\lVert P_\phi v\lVert_{L^2}\geq \lVert e^{\phi/\tilde{h}}(-\tilde{h}^2 \Delta)e^{-\phi/\tilde{h}} v\lVert_{L^2}- \left(\tilde{h}\delta\right)^{1/2}\lVert V-E \lVert_\infty \lVert v\lVert_{L^2}.
		\end{align}
		
		We seek to bound from below the term  $\lVert e^{\phi/\tilde{h}}(-\tilde{h}^2 \Delta)e^{-\phi/\tilde{h}} v\lVert_{L^2}$ by the $L^2$ norm of $v$. The ideal tool for this is Carleman's inequality:
		
		\begin{lem}(Semiclassical Carleman inequality)\\
			Let $U$ be an open set of $\R^2$. Let $P$ be a real elliptic semiclassical operator of $2^{nd}$order with principal symbol $p$. Suppose that the Poisson bracket $\{Re\, p, Im\, p \}>0$ when $p=0$. Then there is $C>0$ and $h_0>0$ such that, for all $v\in \mathcal{C}_{0}^{\infty}(U)$ and for all $0<h<h_0$:
			
			\begin{equation}
				\| P v\|_{L^2}\geq C\,h^{1/2}\|v\|_{H^2} .
			\end{equation}
			
		\end{lem}
		
		For apply this estimate to $ e^{\phi/\tilde{h}}(-\tilde{h}^2 \Delta)e^{-\phi/\tilde{h}}$, we need to find a $\phi$ that meet the requirements. We  would like to have $\phi$ radial but such function has a degenerated critical point at its centre  (Figure~\ref{fig:PointCritiqueFonctionRadiale}), the consequence is that $\{Re\, p, Im\, p \}$ can't be positive when $p=0$. 
		
		\begin{figure}
			\centering
			\includegraphics[width=0.6\linewidth]{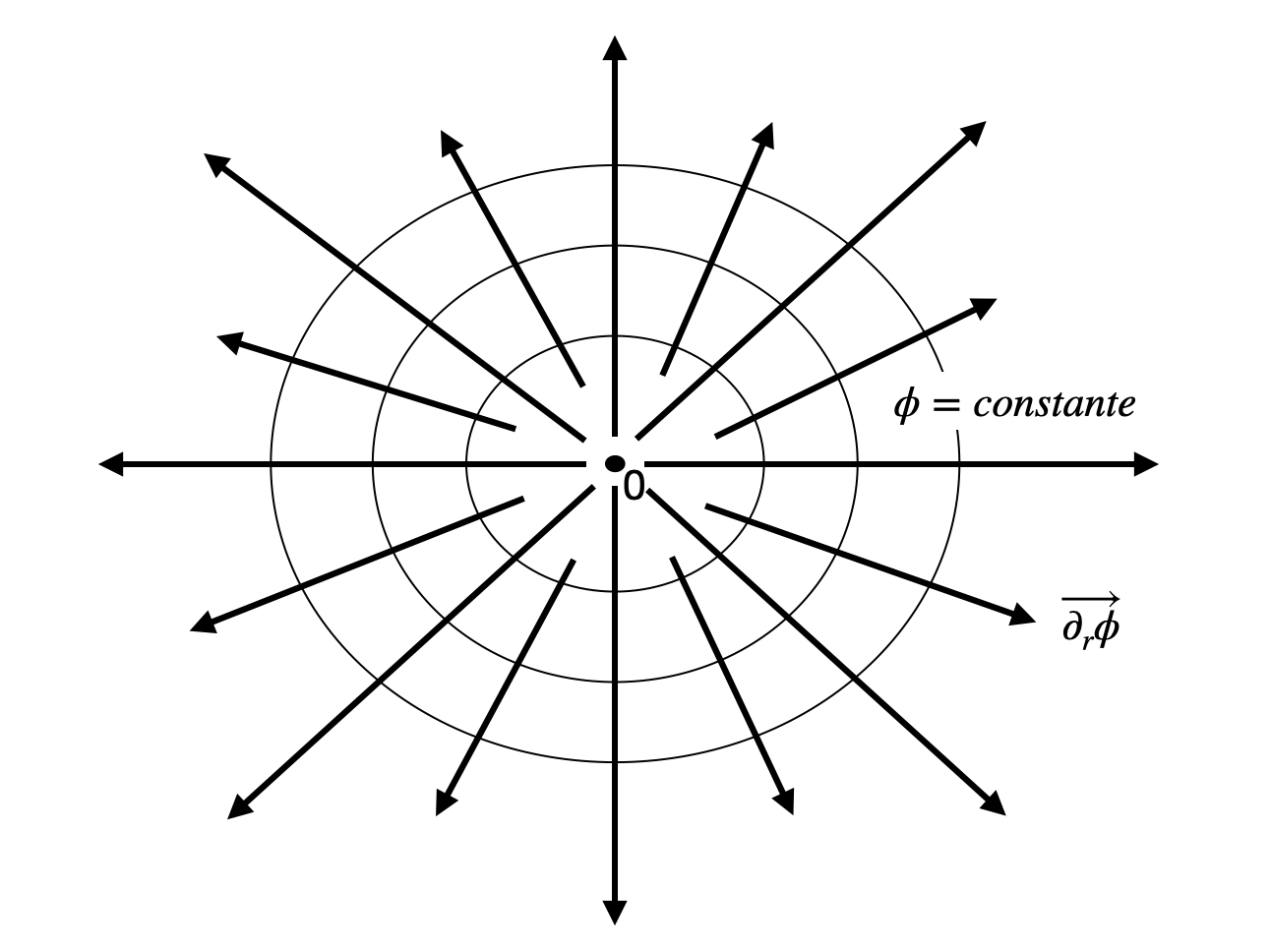}
			\caption{Critical point for a radial function: $\phi$ is a radial function in $0$. $\phi$ is constant on circles centred on $0$ and moves orthogonally to these circles. In $0$ the derivative vector should point in all directions, which is impossible. So $0$ is a critical point of $\phi$.}
			\label{fig:PointCritiqueFonctionRadiale}
		\end{figure}
		
		To solve this problem, let us take two points $x_{1,h}$ and $x_{2,h}$ of $B(0,2h\epsilon(h))$ close enough to $0$ and pose:
		
		\begin{equation}
			\label{eq: phi}
			e^{2\phi(x)/h}= e^{2\phi_{x_{1,h}}(x)/h}\chi_{1,h}^2(x)+e^{2\phi_{x_{2,h}}(x)/h}\chi_{2,h}^2(x)
		\end{equation}
		
		with
		
		\begin{itemize}
			\item $\phi_{x_{1,h}}(x):=\epsilon(h)^{-4/3} \|x-x_{1,h}\|^{4/3}$,
			\item $\phi_{x_{2,h}}(x):=\epsilon(h)^{-4/3} \|x-x_{2,h}\|^{4/3}$,
			\item $ \chi_{1,h}$ zero on a ball $B_{1,h}$ centred on $x_{1,h}$ with radius $h\epsilon(h)/100$ and equal to $1$ on $\mathbb{R}^2\setminus\tilde{B}_{1,h}$ where $B_{1,h}\subset \tilde{B}_{1,h}$,
			\item $\chi_{2,h}$ zero on a ball $B_{2,h}$ centred on $x_{2,h}$ with radius $h\epsilon(h)/100$ and equal to $1$ on $\mathbb{R}^2\setminus\tilde{B}_{2,h}$ where $B_{2,h}\subset  \tilde{B}_{2,h}$,
			\item $\tilde{B}_1 \bigcap \tilde{B}_2 = \emptyset$ (Figure.~\ref{fig:ConstructionDePhi}).
			\vspace{0,5cm}
		\end{itemize}
		
		\begin{figure}
			\centering
			\includegraphics[width=0.6\textwidth]{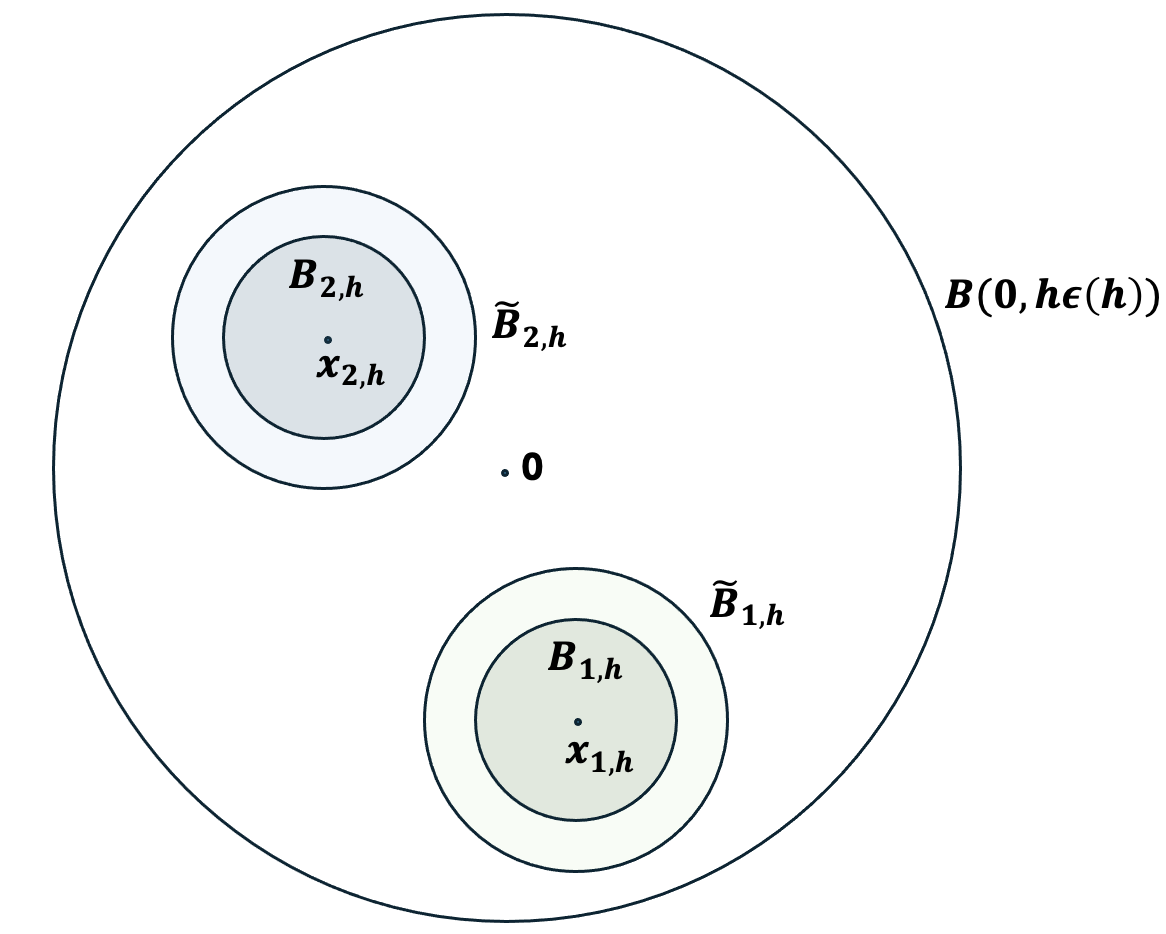}
			\caption{Cutting the ball $B(0,h\epsilon(h))$:  $\chi_{1,h}=0$ on $B_{1,h}$ and $=1$ outside $\tilde{B}_{1,h}$ (green), and $\chi_{2,h}=0$ on $B_{2,h}$ and $=1$ outside $\tilde{B}_{2,h}$ (blue). }
			\label{fig:ConstructionDePhi}
		\end{figure}
		
		With this construction, we are no longer trying to construct one Carleman inequality on a ball, but two on deformed rings,  $B(0,2h\epsilon(h))\setminus B_{1,h}$ and $B(0,2h\epsilon(h))\setminus B_{2,h}$. We will glue them back together later. To obtain the inequality, we adapt a semiclassical proof of the Carleman inequality to our framework.  
		For convenience, the next developments will be made on a ring $A(0,h\epsilon(h)/100,2h\epsilon(h))$ centre on $0$ and $\phi_h(r)= \epsilon(h)^{-4/3} r^{4/3}$. However, the approach and the result will be the same for  $B(0,2h\epsilon(h))\setminus B_{1,h}$ and $B(0,2h\epsilon(h))\setminus B_{2,h}$. 
		
		To begin with, we rewrite $\Delta$ in polar form:
		
		\begin{equation}
			-\tilde{h}^2\Delta= -\tilde{h}^2 \left[\partial_r^2 +\frac{1}{r}\partial_r +\frac{1}{r^2}\Delta_{S^{1}} \right].
		\end{equation}
		
		We expend the ball $B(0,h\epsilon(h))$ by a factor of $[h\epsilon(h)]^{-1}$. So we have the operator:
		
		\begin{equation}
			-\frac{\tilde{h}^2}{h^2\epsilon(h)^2}\left[\partial_r^2+ \frac{1}{r}\partial_r +\frac{1}{r^2}\Delta_{S^{1}}\right]
		\end{equation}
		
		on the ring $A(0,1/100,2)$. To obtain the Carleman inequality we will do some symbolic algebra, for that we quantify in $\delta^{1/3}$. The radial principal symbol of $-\delta^{2/3} \Delta$ is:
		
		\begin{equation}
			p(r,\sigma)= \sigma^2 + \frac{1}{r^2}(\Lambda^*)^2.
		\end{equation}
		
		So by composition with $e^{\pm \phi/\delta^{1/3}}$ ($\phi=r^{4/3}$), the radial principal symbol of $e^{\phi/\delta^{1/3}}(-\delta^{2/3} \Delta)e^{-\phi/ \delta^{1/3}}$ is:
		
		\begin{equation}
			p_{\phi}(r,\sigma)=p(r,\sigma +i\partial_r \phi),
		\end{equation}
		
		that is:
		
		\begin{equation}
			p_{\phi}(r,\sigma)= \left[(\sigma+i\partial_r\phi)^2+\frac{(\Lambda^*)^2}{r^2}\right].
		\end{equation} 
		
		Because of $\partial_r\phi>0$, the principal cancels when: 
		
		\begin{equation}
			 \sigma=0 \quad \mbox{and } (\partial_r\phi)^2=\frac{(\Lambda^*)^2}{r^2}.
		\end{equation}
		
		We want the Poisson bracket $\{Re\, p_{\phi}, Im\, p_{\phi}\}$ to be bounded by below when $p_{\phi}(r,\sigma)=0$.
		
		\begin{align}
			\{Re\, p_{\phi}, Im\, p_{\phi}\}=& 4\left[\sigma^2\partial_r^2\phi +\left(\partial_r^2 \phi \partial_r \phi +\frac{(\Lambda^*)^2}{r^3} \right)\partial_r \phi \right],
		\end{align}
		
		so when $p_{\phi}(r,\sigma)=0$, the Poisson bracket reduces to:
		
		\begin{equation}
			\{Re\, p_{\phi}, Im\, p_{\phi}\}= 4\left(\partial_r^2 \phi \partial_r \phi +\frac{(\partial_r \phi)^2}{r} \right)\partial_r \phi.
		\end{equation}
		
		 $\phi= r^{4/3}$, so we have the bound:
		
		\begin{equation}
			4\left(\partial_r^2 \phi \partial_r \phi +\frac{(\partial_r \phi)^2}{r} \right)\partial_r \phi \geq \partial_r^2 \phi (\partial_r \phi)^2 \geq C>0,
		\end{equation}
		
		therefore the Poisson bracket is positive.
		Then $|p_{\phi}|^2 =O( \langle \sigma \rangle^4=(1+|\sigma|^2)^2)$ when $\sigma$ is large. So we can find $\tilde{C}>0$ and $d>0$ such that:
		
		\begin{equation}
			d|p_{\phi}|^2+ \{Re\, p_{\phi}, Im\, p_{\phi}\} \geq  \tilde{C}\langle \sigma \rangle^4.
		\end{equation}
		
		By the Garding inequality (see for example \cite{alinhac2007pseudo}), we obtain for  $\delta$ small enough:
		
		\begin{align}
			&\frac{1}{\delta^{1/3}} \Biggl\langle \left[\left(e^{\phi/\delta^{1/3}}\left(-\delta^{2/3}\Delta\right)e^{-\phi/\delta^{1/3}}\right)^*,e^{\phi/\delta^{1/3}}\left(-\delta^{2/3}\Delta\right)e^{-\phi/\delta^{1/3}}   \right]v,v \Biggr\rangle \nn\\
			&+d\bigg\|e^{\phi/\delta^{1/3}}(-\delta^{2/3} \Delta)e^{-\phi/\delta^{1/3}}v\bigg\|_{L^2}^2\geq \frac{\tilde{C}}{2}\big\|v \big\|_{H^2}^2
		\end{align}
		
		for $v \in \mathcal{C}_0^{\infty}(A(0,1/100,2))$. With the classic equation:
		
		\begin{align}
			&\frac{1}{\delta^{1/3}} \Biggl\langle \left[\left(e^{\phi/\delta^{1/3}}\left(-\delta^{2/3}\Delta\right)e^{-\phi/\delta^{1/3}}\right)^*,e^{\phi/\delta^{1/3}}\left(-\delta^{2/3}\Delta\right)e^{-\phi/\delta^{1/3}}   \right]v,v \Biggr\rangle\nn\\
			&= \frac{1}{\delta^{1/3}}\Bigg\|e^{\phi/\delta^{1/3}}\left(-\delta^{2/3} \Delta\right)e^{-\phi/\delta^{1/3}}v\Bigg\|_{L^2}^2- \frac{1}{\delta^{1/3}}\Bigg\|\left(e^{\phi/\delta^{1/3}}\left(-\delta^{2/3} \Delta\right) e^{-\phi/\delta^{1/3}}\right)^*v\Bigg\|_{L^2}^2\nn\\
			&\leq \frac{1}{\delta^{1/3}}\Bigg\|e^{\phi/\delta^{1/3}}\left(-\delta^{2/3} \Delta\right)e^{-\phi/\delta^{1/3}}v\Bigg\|_{L^2}^2,			
		\end{align}
		
		 by replacing the $H^2$ norm by a $L^2$ norm and by multiplying the inequality by $h^{4/3}/\epsilon(h)^4$, we obtain this estimate:
		
		\begin{equation}
			\Big\|e^{\phi_h/\tilde{h}}(-\frac{\tilde{h}^2}{h^2\epsilon(h)^2} \Delta)e^{-\phi_h/\tilde{h}}v \Big\|_{L^2}^2\geq \frac{\tilde{h} C_0}{2\epsilon(h)^4}\big\|v \big\|_{L^2}^2.
		\end{equation}
		
		Finally, we come back to the ring $A(0,h\epsilon(h)/100,2h\epsilon(h))$:
		
		\begin{equation}
			\Big\| e^{\phi/\tilde{h}}(-\tilde{h}^2\Delta)e^{-\phi/\tilde{h}}v\Big\|_{L^2}^2\geq \frac{\tilde{h}C_0}{2\epsilon(h)^4}\big\|v \big\|_{L^2}^2
		\end{equation}
		
		for $v \in \mathcal{C}_0^{\infty}(A(0,h\epsilon(h)/100,2h\epsilon(h)))$.
	Then using the inequality \ref{eq: inég carleman V Linfini}, we get:
	
	\begin{align}
		\big\lVert P_\phi v \big\lVert_{L^2}&\geq \Big\lVert e^{\phi/\tilde{h}}(-\tilde{h}^2 \Delta)e^{-\phi/\tilde{h}} v \Big\lVert_{L^2}- \left(\tilde{h}\delta\right)^{1/2}\big\lVert V-E \big\lVert_\infty \big\lVert v \big\lVert_{L^2}\nn\\
		&\geq \frac{C_0\tilde{h}^{1/2}\delta^{1/6}}{\epsilon(h)^2 }\big\lVert v \big\lVert_{L^2}
	\end{align}
	
	for $\delta$  small enough. And so, we get with $v=e^{\phi/\tilde{h}}u$:
	
	\begin{equation}
		\left(h\delta\right)^{2/3}\Big\lVert e^{\phi/\tilde{h}}(-h^2\Delta+V-E) u\Big\lVert_{L^2}\geq \frac{C_0 h^{2/3}\delta^{1/3}}{\epsilon(h)^2}\Big\lVert e^{\phi/\tilde{h}}u \Big\lVert_{L^2},
	\end{equation}
	
	That is:
	
	\begin{prop}
		\label{prop: Carleman_avant_recollement}
		For $u \in\mathcal{C}_0^{\infty} (A(0, h\epsilon(h)/100, 2h\epsilon(h)))$:
		
		\begin{equation}
			\frac{\epsilon(h)^2\delta^{1/3}}{C_0}\Big\lVert e^{\phi/\tilde{h}}(-h^2\Delta+V-E) u\Big\lVert_{L^2}\geq \Big\lVert e^{\phi/\tilde{h}}u\Big\lVert_{L^2}.
		\end{equation}
		
	\end{prop}
	
	Finally, based on the result of the proposition \ref{prop: Carleman_avant_recollement}, which also applies to $B(0,2h\epsilon(h))\setminus B_{1,h}$ and $B(0,2h\epsilon(h))\setminus B_{2,h}$,  the resulting inequalities must be put back together and we must remove the assumption of compact support. To do this, we proceed with $\chi_{h} u$ instead of $u$ where $\chi_{h}$ is a stage function on $B(0,\frac{4}{3}h\epsilon(h))$ supported in $B(0,\frac{5}{3}h\epsilon(h))$ and $u$ is smooth. By definition of $\phi$ (see equation~\ref{eq: phi}):
	
	\begin{align}
		\int_{B(0,2h\epsilon(h))} e^{2\phi(x)/\tilde{h}} |\chi_{h}(x) u(x)|^2 dx &= \int_{B(0,2h\epsilon(h))} e^{2\phi_{x_{1,h}}(x)/\tilde{h}} \chi_{1,h}^2(x) |\chi_{h}(x) u(x)|^2 dx \nn\\
		&\quad +\int_{B(0,2h\epsilon(h))} e^{2\phi_{x_{2,h}}(x)/\tilde{h}} \chi_{2,h}^2(x) |\chi_{h}(x) u(x)|^2 dx\nn\\
		(\mbox{prop \ref{prop: Carleman_avant_recollement}})\leq \left(\frac{\epsilon(h)^2\delta^{1/3}}{C_0}\right)^2  \int_{B(0,2h\epsilon(h))} & e^{2\phi_{x_{1,h}}(x)/\tilde{h}}  |(-h^2\Delta +V(x)-E)\chi_{1,h}(x) \chi_{h}(x)u(x)|^2 dx   \nn\\
		\quad +\left(\frac{\epsilon(h)^2\delta^{1/3}}{C_0}\right)^2\int_{B(0,2h\epsilon(h))} & e^{2\phi_{x_{2,h}}(x)/\tilde{h}}  |(-h^2\Delta +V(x)-E)\chi_{2,h}(x) \chi_{h}(x)u(x)|^2 dx   \nn\\
		\leq \left(\frac{\epsilon(h)^2\delta^{1/3}}{C_0}\right)^2  \int_{B(0,2h\epsilon(h))} & e^{2\phi_{x_{1,h}}(x)/\tilde{h}} \chi_{1,h}^2(x) \chi_{h}^2(x) |(-h^2\Delta +V(x)-E)u(x)|^2 dx   \nn\\
		\quad +\left(\frac{\epsilon(h)^2\delta^{1/3}}{C_0}\right)^2  \int_{B(0,2h\epsilon(h))} & e^{2\phi_{x_{2,h}}(x)/\tilde{h}} \chi_{2,h}^2(x) \chi_{h}^2(x) |(-h^2\Delta +V(x)-E)u(x)|^2 dx   \nn\\
		\quad +\left(\frac{\epsilon(h)^2\delta^{1/3}}{C_0}\right)^2  \int_{B(0,2h\epsilon(h))} & e^{2\phi_{x_{1,h}}(x)/\tilde{h}}  |[-h^2\Delta, \chi_{h}\chi_{1,h}]u(x)|^2 dx   \nn\\
		\quad +\left(\frac{\epsilon(h)^2\delta^{1/3}}{C_0}\right)^2  \int_{B(0,2h\epsilon(h))} & e^{2\phi_{x_{2,h}}(x)/\tilde{h}}  |[-h^2\Delta, \chi_{h}\chi_{2,h}]u(x)|^2  dx.  \nn\\
	\end{align}
	
	By using a density argument, we can replace $u$ by $u_h$ our solution to $(-h^2\Delta +V(x)-E)u_h=0$. We get:
	
	\begin{align}
	&\left(\frac{\epsilon(h)^2\delta^{1/3}}{C_0}\right)^2  \int_{B(0,2h\epsilon(h))}  e^{2\phi_{x_{1,h}}(x)/\tilde{h}} \chi_{1,h}^2(x) \chi_{h}^2(x) |(-h^2\Delta +V(x)-E)u_h(x)|^2 dx=0,  \nn\\
	&\left(\frac{\epsilon(h)^2\delta^{1/3}}{C_0}\right)^2  \int_{B(0,2h\epsilon(h))}  e^{2\phi_{x_{2,h}}(x)/\tilde{h}} \chi_{2,h}^2(x) \chi_{h}^2(x) |(-h^2\Delta +V(x)-E)u_h(x)|^2 dx=0.   \nn\\
	\end{align}
	
	For $j=1\, \mbox{or}\, 2$:
	\begin{align}
		|[-h^2\Delta, \chi_{h}\chi_{j,h}]u(x)|^2 &\leq C_j (|h^2\Delta(\chi_{j,h}\chi_{h})u|^2 +|h\nabla (\chi_{j,h}\chi_{h})h\nabla u|^2)\nn\\
		&\leq C_j(|h^2\Delta(\chi_{j,h}\chi_{h})u|^2 +h^2||\nabla (\chi_{j,h}\chi_{h})||_{L^\infty}^2|h\nabla u|^2)
	\end{align}
	
	with $C_j>0$.
	From the Cacciopoli lemma (prop~\ref{prop: Caccio}) applied to $\nabla u_h$ with a cut-off function on the support of $\nabla (\chi_{i,h}\chi_{h})$ which is in $(\tilde{B}_{i,h}\setminus B_{i,h})\bigcup (B(0,\frac{5}{3}h\epsilon(h))\setminus B(0,\frac{4}{3}h\epsilon(h)))$,  we can assume that:
	
	\begin{align}
		&\left(\frac{\epsilon(h)^2\delta^{1/3}}{C_0}\right)^2  \int_{B(0,2h\epsilon(h))} e^{2\phi_{x_{i,h}}(x)/\tilde{h}}  |[-h^2\Delta, \chi_{h}\chi_{i,h}]u(x)|^2  dx  \nn\\
		&\leq\left(\frac{\epsilon(h)^2}{C_0}\right)^2\int_{\mathcal{A}_{i,h}\cup A(0,h\epsilon(h),2h\epsilon(h))} \frac{C}{\epsilon(h)^4}| u(x)|^2e^{2\frac{\phi_{x_{i,h}}(x)}{\tilde{h}}} dx   
	\end{align}
	
	where $\mathcal{A}_{i,h}=A(x_{i,h},r_{1,i,h},r_{2,i,h})$ is a ring centre on $x_{i,h}$ with radius small proportionally to $h\epsilon(h)$ containing the support of $\nabla\chi_{i,h}$ (we can choose $\chi_{j,h}$ such that we can take $A(x_{j,h},r_{1,j,h},r_{2,j,h})=(\tilde{B}_{j,h}\setminus B_{j,h})$) and $C >0$. We get:
	
	\begin{align}
		\int_{B(0,2h\epsilon(h))} e^{2\phi(x)/\tilde{h}} |\chi_{h}(x) u_h(x)|^2 dx &= \int_{B(0,2h\epsilon(h))} e^{2\phi_{x_{1,h}}(x)/\tilde{h}} \chi_{1,h}^2(x) |\chi_{h}(x) u_h(x)|^2 dx \nn\\
		&+\int_{B(0,2h\epsilon(h))} e^{2\phi_{x_{2,h}}(x)/\tilde{h}} \chi_{2,h}^2(x) |\chi_{h}(x) u_h(x)|^2 dx\nn\\
		&\leq  \left(\frac{C}{C_0}\right)^2\int_{A(x_{1,h},r_{1,1,h},r_{2,1,h})} e^{2\phi_{x_{1,h}}(x)/\tilde{h}}  |u_h(x)|^2 dx   \nn\\
		&+   \left(\frac{C}{C_0}\right)^2\int_{A(x_{2,h},r_{1,2,h},r_{2,2,h})} e^{2\phi_{x_{2,h}}(x)/\tilde{h}}  |u_h(x)|^2  dx   \nn\\
		&+ \left(\frac{C}{C_0}\right)^2\int_{ A(0,h\epsilon(h),2h\epsilon(h))} e^{2\phi_{x_{1,h}}(x)/\tilde{h}}  |u_h(x)|^2 dx   \nn\\
		&+   \left(\frac{C}{C_0}\right)^2\int_{ A(0,h\epsilon(h),2h\epsilon(h))} e^{2\phi_{x_{2,h}}(x)/\tilde{h}}  |u_h(x)|^2  dx.   \nn\\
	\end{align}
	
	With our choice of $\chi_{1,h}$ and $\chi_{2,h}$, then we have $A(x_{1,h},r_{1,1,h},r_{2,1,h})\subset Supp(\chi_{2,h})$ and\\
	$A(x_{2,h},r_{1,2,h},r_{2,2,h})\subset Supp(\chi_{1,h})$. Moreover, $\phi_{x_{1,h}}$ is smaller than $\phi_{x_{2,h}}$ on $A(x_{1,h},r_{1,1,h},r_{2,1,h})$ and $\phi_{x_{2,h}}$ is smaller than $\phi_{x_{1,h}}$ on $A(x_{2,h},r_{1,2,h},r_{2,2,h})$. By taking $\delta$ small enough, we obtain: 
	
	\begin{align}
		&\left(\frac{C}{C_0}\right)^2\int_{ A(x_{1,h},r_{1,1,h},r_{2,1,h})} e^{2\phi_{x_{1,h}}(x)/\tilde{h}}  |u_h(x)|^2 dx   \nn\\
		&\leq \frac{1}{2} \int_{B(0,2h\epsilon(h))} e^{2\phi_{x_{2,h}}(x)/\tilde{h}} \chi_{2,h}^2(x) |\chi_{h}(x) u_h(x)|^2 dx \nn\\
	\end{align}
	
	and
	
	\begin{align}
		&\left(\frac{C}{C_0}\right)^2\int_{ A(x_{2,h},r_{1,2,h},r_{2,2,h})} e^{2\phi_{x_{2,h}}(x)/\tilde{h}}  |u_h(x)|^2 dx   \nn\\
		&\leq \frac{1}{2} \int_{B(0,2h\epsilon(h))} e^{2\phi_{x_{1,h}}(x)/\tilde{h}} \chi_{2,h}^2(x) |\chi_{h}(x) u_h(x)|^2 dx. \nn\\
	\end{align}
	
	Thus:
	
	\begin{align}
		\int_{B(0,h\epsilon(h))} e^{2\phi(x)/\tilde{h}} |u_h(x)|^2 dx\leq
		& 2\left(\frac{C}{C_0}\right)^2\int_{ A(0,h\epsilon(h),2h\epsilon(h))} e^{2\phi_{x_{1,h}}(x)/\tilde{h}}  |u_h(x)|^2 dx   \nn\\
		& + 2\left(\frac{C}{C_0}\right)^2\int_{ A(0,h\epsilon(h),2h\epsilon(h))} e^{2\phi_{x_{2,h}}(x)/\tilde{h}}  |u_h(x)|^2 dx. 
	\end{align}
	
	Finally, since $|x|$ is bounded by $2h\epsilon(h)$, we have \\
	$1\leq e^{2\phi_{x_{1,h}}(x)/\tilde{h}},e^{2\phi_{x_{2,h}}(x)/\tilde{h}},e^{2\phi(x)/\tilde{h}} 
	\leq e^{(4h\epsilon(h))^{4/3}/\epsilon(h)^{4/3}h^{4/3}\delta^{1/3}} \leq C(\delta)$. $\delta$ is independent of $h$, so we can fix it to obtain the desired inequality: 
	
	\begin{equation}
		\int_{B(0,h\epsilon(h))} | u_h(x)|^2 dx \leq C \int_{ A(0,h\epsilon(h),2h\epsilon(h))}  |u_h(x)|^2 dx   
	\end{equation}
	
	with $C>0$.
	
\end{proof}

\bibliographystyle{elsarticle-harv} 
\typeout{}
\bibliography{references}

\end{document}